\documentclass[11pt,reqno]{amsart}
\usepackage{graphicx}
\usepackage{amscd,amsmath,amsopn,amssymb,amsthm,multicol}
\usepackage[color,matrix, all, 2cell]{xy}
\usepackage{amscd}
\usepackage{lscape}
\usepackage{slashed}
\usepackage{graphicx}
\usepackage{setspace}
\usepackage{upgreek}
\usepackage{textgreek}
\usepackage{enumerate}
\usepackage{color}
\usepackage{lscape}
\usepackage{tikz}
\usepackage{multirow}
\usepackage{cancel}
\usepackage{soul}
\usepackage{harmony}
\usepackage{comment}
\usepackage{wasysym}
\usepackage{mathrsfs}
\usepackage{mathtools}
\usepackage{tcolorbox}
\usepackage{tikz}

 \numberwithin{equation}{section}
\DeclareMathOperator{\dd}{d}
\DeclareMathOperator{\Ric}{\mathsf{Ric}}
\DeclareMathOperator{\ric}{\mathsf{ric}}
\DeclareMathOperator{\vol}{\mathsf{vol}}

\DeclareMathOperator{\R}{\mathbb{R}}

\DeclareMathOperator{\Flu}{\mathsf{F}}
\DeclareMathOperator{\X}{\mathsf{X}}
\DeclareMathOperator{\Y}{\mathsf{Y}}
\DeclareMathOperator{\M}{\mathsf{M}}

\DeclareMathOperator{\SO}{\mathsf{SO}}

 \DeclareMathOperator{\SU}{\mathsf{SU}}

\DeclareMathOperator{\E}{\mathsf{E}}
\DeclareMathOperator{\Ss}{\mathsf{S}}

\definecolor{dblue}{rgb}{0.01,0.01,0.44}
\definecolor{red}{rgb}{0.57,0.11,0.15}

\newcommand{\al}{\alpha}
\newcommand{\be}{\beta}
\newcommand{\eps}{\epsilon}
\newcommand{\bb}{\mathbb}
\newcommand{\wi}{\widetilde}

\theoremstyle{plain}
\newtheorem{lemma}{Lemma} [section]
\newtheorem{theorem}[lemma]{Theorem}
\newtheorem{corol}[lemma] {Corollary}
\newtheorem{prop} [lemma]{Proposition}

\theoremstyle{definition}
\newtheorem{definition}[lemma] {Definition}
\newtheorem{example}[lemma] {Example}

\newtheorem{remark}[lemma] {Remark}
\newtheorem*{remark*}{Remark}

\definecolor{dark}{rgb}{0.18,0.18,0.68}
\definecolor{mydark}{rgb}{0.78,0.08,0.08}
\definecolor{crew}{rgb}{0.2,0.5,0.2}
\definecolor{mmg}{rgb}{0.31,0.50,0.23}
\definecolor{dblue}{rgb}{0.01,0.01,0.44}
\definecolor{red}{rgb}{0.57,0.11,0.15}
\definecolor{cobalt}{RGB}{12,89,188}
\usepackage[colorlinks,citecolor=cobalt,linkcolor=cobalt,urlcolor=cobalt,pdfpagemode=UseNone,backref = page]{hyperref}

\usepackage{tcolorbox}
\definecolor{mycolor}{rgb}{0.122, 1.435, 0.798}

\newtcbox{\mybox}{on line,
  colframe=mycolor,colback=mycolor!10!white,
  boxrule=0.5pt,arc=4pt,boxsep=0pt,left=6pt,right=6pt,top=6pt,bottom=6pt}
  
  \definecolor{mycol}{rgb}{0.422, 0.498, 0.135}
  \newtcbox{\myboxx}{on line,
  colframe=mycol,colback=mycol!10!white,
  boxrule=0.5pt,arc=4pt,boxsep=0pt,left=6pt,right=6pt,top=6pt,bottom=6pt}
  
  \usepackage{etoolbox}

\makeatletter
\patchcmd{\@setauthors}{\MakeUppercase}{}{}{}
\makeatother

 \input ulem.sty
 
\begin{document}


\title[Decomposable $(5, 6)$-solutions in eleven-dimensional supergravity]{Decomposable $(5, 6)$-solutions in eleven-dimensional supergravity}

\author[]{Hanci Chi\textsuperscript{1,\MakeLowercase{a})}, Ioannis Chrysikos\textsuperscript{2,\MakeLowercase{b})}, Eivind Schneider\textsuperscript{2,\MakeLowercase{c})} \\ 
\\
\textsuperscript{1}Xi'an Jiaotong-Liverpool University, 111 Ren'ai Road, Suzhou Industrial Park, \\Suzhou, Jiangsu Province, P. R. China. \\ 
\textsuperscript{2}Faculty of Science, University of Hradec Kr\'alov\'e, Rokitansk\'eho 62, \\ Hradec Kr\'alov\'e
50003, Czech Republic.
}

\begin{abstract}
We present decomposable (5,6)-solutions $\widetilde{M}^{1,4} \times M^6$ in eleven-dimensional supergravity by solving the bosonic supergravity equations for a variety of non-trivial flux forms. Many of the bosonic backgrounds presented here are induced by various types of null flux forms on products of certain totally Ricci-isotropic Lorentzian Walker manifolds and Ricci-flat Riemannian manifolds. These constructions provide an analogue of work done by I. Chrysikos and A. Galaev who made similar computations  for decomposable (6,5)-solutions. We also present bosonic backgrounds that are products of Lorentzian Einstein manifolds with negative Einstein constant (in the ``mostly plus'' convention) and Riemannian K\"ahler-Einstein manifolds with positive Einstein constant. This conclusion generalizes a result of C. N. Pope and P. van Nieuwenhuizen concerning the appearance of six-dimensional K\"ahler-Einstein manifolds in eleven-dimensional supergravity. In this setting we construct infinitely many non-symmetric decomposable (5, 6)-supergravity backgrounds by using the infinitely many Lorentzian Einstein-Sasakian structures with negative Einstein constant on the 5-sphere, known from the work of C. P. Boyer et al. 
\end{abstract}

\thanks{E-mail:  \textsuperscript{a)}{\tt hanci.chi@xjtlu.edu.cn}\hspace{1pt}, \textsuperscript{b)}{\tt chrysikos@math.muni.cz}\hspace{1pt}, \textsuperscript{c)}{\tt eivind.schneider@uit.no}\hspace{1pt}. \\
This article may be downloaded for personal use only. Any other use requires prior permission of the author and AIP Publishing. This article appeared in J. Math. Phys. 64, 062301 (2023) and may be found at \url{https://doi.org/10.1063/5.0142572}.}

\maketitle





\section*{Introduction}

\addtocontents{toc}{\protect\setcounter{tocdepth}{1}}
\subsection*{Motivation}

The five established ten-dimensional superstring theories (Type I, Type IIA, Type IIB, heterotic $\SO(32)$ and heterotic $\E_8\times\E_8$) provide frameworks for uniting quantum theory and general relativity. By string dualities such as T-duality, a unique eleven-dimensional superstring theory, called {\it $\M$-theory}, unites these five theories. As a result, eleven-dimensional supergravity theory, viewed as a low-energy limit of $\M$-theory, has attracted much attention during the last half century (see for example \cite{DNP86, Wi95, Duff99, GSW12, T14}). 

The Lagrangian of eleven-dimensional supergravity was proposed in \cite{Cre}. The fields in the theory is a Lorentzian metric $h$, a closed 4-form $\Flu$ (called the flux form), and a Majorana spinor $\Psi$. They are defined on an eleven-dimensional manifold $\X$ and are subject to the equations of motion determined by the Lagrangian. A special class of supergravity solutions are those with vanishing fermionic part, $\Psi=0$. In this case the equations of motion reduce to a simpler set of equations, involving only $h$ and $\Flu$, which we call the bosonic supergravity equations. This set of equations closely resembles the Einstein-Maxwell equations in four dimensions. 
Solutions to the bosonic supergravity equations are called bosonic supergravity backgrounds. The bosonic backgrounds include the special class of eleven-dimensional Ricci-flat Lorentzian manifolds for which $\Flu=0$. 

Finding bosonic supergravity backgrounds is an important task, and the literature on bosonic supergravity backgrounds is vast. Several geometrical tools and constructions have been used for finding them, including manifolds with special holonomy or special $G$-structures,  irreducible symmetric spaces, compactifications, Killing superalgebras,   certain ansatzes on $h$ and $\Flu$, and other.
 We  refer  to some representative works
  \cite{ PT95,  BB96,   F00, Fig1,  BDS02,  Gaun1, Mart, HM05, Ts06, F07, Fig2, GP15,  MFS18, Fei, BF21}, and the reader can find more references therein.

Supersymmetries also play an important role in these investigations. The maximally supersymmetric bosonic backgrounds are, in addition to flat Minkowski space, the Freund-Rubin backgrounds $(\mathsf{AdS})_7\times \Ss^4$  and $(\mathsf{AdS})_4\times \Ss^7$  (\cite{FR80}) and a particular pp-wave (see for example \cite{FOFP}). These are locally homogeneous, something which is true for all backgrounds admitting more than half of the maximal amount of supersymmetries (\cite{FH12}). Other well-known examples are the  M2-brane and the M5-brane (\cite{DS91,Gu92}) whose near-horizon geometries are the  Freund-Rubin backgrounds (see also \cite{F98,S09}). These have exactly half of the maximal amount of supersymmetries. On the other side of the spectrum, with respect to the number of supersymmetries admitted, we have bosonic backgrounds such as $(\mathsf{AdS})_5\times \mathbb{CP}^3$ which admit no supersymmetry (see \cite{Pope89}).

  \subsection*{Outline}

In this article we search for bosonic supergravity backgrounds that are products of an oriented Lorentzian manifold $(\widetilde{M}^{1,4},\tilde g)$ and an oriented  Riemannian manifold $(M^6, g)$,  with flux form $\Flu \in\Omega^4(\mathsf{X})$  of the type 
\[
  \Flu= \varphi \tilde \al +   \tilde \be \wedge \nu +  \tilde \gamma \wedge \delta  +  \tilde \varpi  \wedge \eps  + \tilde \psi \theta\,,
\]
where the $i$th term is the product of an $(5-i)$-form on $\wi{M}^{1,4}$ and an $(i-1)$-form on $M^6$, for $i=1,\ldots,5$. For various flux forms of the above type, we write down the corresponding simplified form of the bosonic supergravity equations and find particular solutions to these equations. Our work can be considered as a natural continuation of \cite{CG}, where products of  six-dimensional Lorentzian manifolds and  five-dimensional Riemannian manifolds  are treated in a similar way.

We begin by describing  the general constraints that appear due to  the bosonic supergravity equations (which we split up into the closedness   condition, the Maxwell equation and the supergravity Einstein equation).  For a general 4-form $\Flu$ of the above form, the resulting system is still quite complicated. See for example Proposition \ref{MAX} for the Maxwell equation and the equations (\ref{HHRic}), (\ref{VVRic}), (\ref{VHRic})  for the  supergravity Einstein equation. In order to obtain a more tractable system of equations,  we specify $\Flu$ even further by letting three or four  of its terms vanish.   Then, as in \cite{CG}, the constraints which occur due to the Maxwell equation in combination with the closedness  condition are simplified (see Proposition \ref{casesMC}), and the same applies for the supergravity Einstein equation  (Proposition \ref{spck}). 
It is worth mentioning that the form of $\Flu$ can impose non-trivial restrictions on the geometry of  $\wi M^{1, 4}$ or  $M^6$ (see Corollary \ref{constantlength}).  For example, for $\Flu=\tilde \alpha$ the supergravity Einstein equation implies that $(M^6,g)$ is an Einstein manifold, while for $\Flu=\theta$ the Einstein equation implies that $(\wi{M}^{1,4},\tilde g)$ is an Einstein manifold. In both  cases, the scalar curvature of $(\X,h)$ is constant (Corollary \ref{useforflags}).  

In order to find explicit solutions of eleven-dimensional bosonic supergravity, we take two  different approaches. 
First, we examine the case when $\Flu$ is composed of null forms. In this case, the bosonic supergravity equations simplify significantly, as shown in Proposition \ref{GeneralNull} and Theorems \ref{th:VVRicalpha}, \ref{th:VVricbeta}, \ref{th:VVricGamma}, \ref{th:VVricvarpi}, \ref{alpha-beta-nu} and \ref{beta-nu-varpi-epsilon}. Moreover, the supergravity Einstein equation requires $(M^6,g)$ to be Ricci-flat (Proposition \ref{GeneralNull}). In addition, for a bosonic supergravity background $(\X^{1,10}=\widetilde{M}^{1,4} \times M^6, h=\tilde{g}+g,\Flu=\tilde \varpi \wedge \epsilon)$, where $\tilde\varpi \in \Omega ^1(\wi{ M}^{1,4})$ is null,  we see in Corollary \ref{cor:RicIso} that $(\X, h)$ is totally Ricci-isotropic, as an analog of  \cite[Cor. 4.10]{CG}. To find concrete solutions to the bosonic supergravity equations, we follow \cite{CG} and assume that the Lorentzian part $(\widetilde{M}^{1,4},\tilde g)$ is a special type of Walker metrics, since these come equipped with a distribution of null lines, from which non-trivial null flux forms can be built.  Propositions \ref{propalpha}, \ref{propbeta}, \ref{propgamma}, \ref{propvarpi}, \ref{gen1}, \ref{gen2} and \ref{gen3} concern  non-symmetric  bosonic backgrounds that are  direct products of a Ricci-isotropic Lorentzian Walker manifold and a Ricci-flat Riemannian manifold.  These products  have special holonomy properties and can potentially support supersymmetries (see also \cite{F00}). We explicitly illustrate these results by examples involving five-dimensional pp-waves, while an investigation of supersymmetries will be left for a forthcoming work.


In the second approach, we study some cases where the Riemannian part $(M^6,g, \omega)$ is a   K\"ahler manifold. In this case, we do not assume that $\Flu$ is null, but rather that it is related to the K\"ahler form $\omega$. In particular, we consider the cases $\Flu= \tilde \gamma \wedge \delta$ with $\delta=\omega$ and $\Flu=\theta=c \star_6 \omega$, where $c$ is a constant. We see that if the flux form is given by $\Flu=\tilde \gamma \wedge \delta$ and $\|\tilde \gamma\|_{\tilde g}^2$ is not constant, then the supergravity Einstein equation forces $M^6$ to be a Ricci-flat almost Hermitian manifold (Corollary \ref{constantlength}), so in Proposition \ref{prop61} we write down the bosonic supergravity equations for the case when $M^6$ is a K\"ahler manifold. For the case where the flux form is given by $\Flu=c\star_6 \omega$, Proposition \ref{prop:thetakahler} says that the bosonic supergravity equations are satisfied, if and only if both $(\widetilde{M}^{1,4},\tilde g)$ and $(M^6,g)$ are Einstein with Einstein constants $\frac{1}{6} c^2$ and $-\frac{1}{6} c^2$, respectively. Thus $(\widetilde{M}^{1,4},\tilde g)$ has positive scalar curvature, while  $(M^6,g)$ has negative scalar curvature. For instance, the symmetric spaces $\mathbb{CP}^3$ and $\mathsf{Gr}_{+}(2,5)$ endowed with their respective (unique) homogeneous K\"ahler-Einstein metrics can be used to obtain some of the  decomposable symmetric supergravity backgrounds presented in \cite{Fig2}. Note that in this paper we use the ``mostly minus'' convention for the Lorentzian metric $h$, and thus  a Riemannian metric $g$ is viewed as a negative definite metric.
 Proposition \ref{prop:thetakahler} generalizes a result presented in \cite{Pope89} which involved bosonic backgrounds of the form $(\mathsf{AdS})_{5}\times M^6$, where $M^6$ is a compact K\"ahler manifold.  
We discuss some possible candidates for the Einstein manifold $(\widetilde{M}^{1,4},\tilde g)$, other than $(\mathsf{AdS})_{5}$. 
In particular, we are based on  {\it  negative   Sasakian geometries} and use the infinitely many different Lorentzian Einstein-Sasakian structures with negative Einstein constant (in the ``mostly plus'' convention) described on the 5-sphere $\Ss^5$  by Boyer  et al. \cite{Boyer}.  In this way we get infinitely many new bosonic non-symmetric decomposable $(5, 6)$-solutions in eleven-dimensional supergravity given by  $\Ss^5\times M^6$, 
 where $M^6$ is any six-dimensional (de Rham irreducible) K\"ahler-Einstein manifold with positive scalar curvature (also in the ``mostly plus'' convention)  and $\Ss^5$ is endowed with one of the Lorentzian Einstein-Sasakian structures mentioned above.  Note that all such solutions that are based on the same K\"ahler-Einstein manifold $M^6$ have equal flux forms. Other such examples can be  obtained by using  the connected sum $\sharp k (\Ss^2\times\Ss^3)$, since this manifold  also admits Lorentzian Einstein-Sasakian metrics for any integer $k\geq 1$.

We should finally mention that the above conclusion fails if $M^6$ is a six-dimensional (strictly) nearly K\"ahler manifold, since in this case the K\"ahler form $\omega$ is not closed, so the 4-form $\Flu$ indicated above cannot serve as a flux form (see the final section).  As a consequence,  and in line with the conclusion pointed out in \cite{Pope89} for $(\mathsf{AdS})_5\times M^6$,  bosonic solutions of the form $\widetilde{M}^{1,4} \times M^6$, where   $\widetilde{M}^{1,4}$ is a  Lorentzian Einstein manifold   and $M^6$ is a compact K\"ahler-Einstein manifold,  are not expected to  admit supersymmetries. Essentially, this is because  in dimension 6 smooth spin manifolds admitting  real Killing spinors are exhausted by  nearly K\"ahler manifolds, see for example  \cite{Baum}, and  see also \cite{Gaun2} for the classification of  eleven-dimensional superymmetric  supergravity solutions containing $(\mathsf{AdS})_{5}$. 


The paper is structured as follows. 
In Section \ref{preliminaries} we lay out the framework that will be used throughout the paper and establish some notation. We introduce the eleven-dimensional bosonic supergravity equations, and write down the ansatz of the general flux form, which we use in this paper.  
The bosonic supergravity equations corresponding to this  ansatz are computed in Sections \ref{closuremaxwell} and \ref{einstein}, respectively. There we also investigate the form of the equations after further simplification of the flux form, and state some general consequences of the equations. 
As mentioned above, the bosonic supergravity equations are simpler when the flux form is composed of null forms, and in Section \ref{null} we present some general results for such flux forms.  Next, in  Section \ref{walker} we apply these results to Ricci-isotropic Lorentzian Walker manifolds and produce several explicit examples of decomposable (5, 6)-supergravity backgrounds. 
In Section \ref{other} we drop the requirement that $\Flu$ is null and analyze the appearance of K\"ahler-Einstein manifolds and of negative Einstein-Sasakian geometries in our decomposable (5, 6)-solutions.

\addtocontents{toc}{\protect\setcounter{tocdepth}{2}}
\section{Preliminaries} \label{preliminaries}
In this work we study  connected eleven-dimensional Lorentzian manifolds of the form 
\[
\X^{1, 10} =   \wi M^{1, 4}\times M^6 \,,
\]
where $(\wi M^{1, 4}, \tilde g )$ is a five-dimensional connected oriented  Lorentzian manifold  and $(M^{6}, g)$ is a six-dimensional connected oriented Riemannian manifold. Our aim is to present on such products a systematic examination of  the  bosonic supergravity  equations, i.e. of the following system of field equations (see for example \cite{Fig1,ACT}):
\begin{equation}\left\{
\begin{tabular}{l l l}
$\dd \Flu$ &$=$&0\,, \\
$\dd \star \Flu$ &$=$ &$\frac{1}{2}\Flu \wedge \Flu$\,, \\
$\Ric_{h}(X,Y)$&$=$&$-\frac{1}{2} \langle X\lrcorner \Flu, Y\lrcorner \Flu\rangle_{h} +\frac{1}{6}h(X,Y) \|\Flu\|^2_{h}$\,.
\end{tabular}\right. \label{MainEq}
\end{equation}
Here,  the Lorentzian metric on $\X^{1, 10}$ is the product metric  $h =  \tilde g+g$ and 
$\star \colon \Omega^{k}(\X^{1,10})\to \Omega^{11-k}(\X^{1,10})$ is the Hodge-star operator on $(\X^{1, 10}, h)$, defined by
$
\al \wedge \star\be = \langle \al , \be\rangle_{h}\vol _{\X},
$
where $\vol_{\X}=\vol_{\wi{M}}\wedge\vol_{M}$  denotes the volume form on $(\X^{1,10}, h)$.  We also have $\|\Flu\|_h^2=\langle \Flu,\Flu \rangle_h$. The bosonic field  $\Flu$ is a global 4-form on $\X^{1, 10}$,  called the {\it flux form} which, together with the Lorentzian metric $h$,  form the bosonic sector of eleven-dimensional supergravity.  We will refer to the three conditions  appearing in (\ref{MainEq}) as the {\it closedness condition}, the   {\it Maxwell equation}, and the  {\it supergravity Einstein equation}, respectively.
Triples $(\X^{1, 10}, h, \Flu)$ solving this system of equations are called {\it   bosonic supergravity backgrounds}. 

%

\begin{remark}
In this paper we apply the ``mostly minus'' convention. That is, the signature for $h$ is $(+,-,\dots,-)$ and hence $g$ is a negative definite Riemannian metric on $M$. Recall that for any two $k$-forms $\omega$ and $\phi$, we have 
$$
\langle\omega,\phi\rangle_h= \frac{1}{k!}\sum_{1\leq i_\alpha,j_\beta\leq 11}\omega_{i_1\dots i_k}\phi_{j_1\dots j_k}h^{i_1j_1}\dots h^{i_kj_k}
$$
Then, for a $k$-form $\omega\in \Omega^k(M^6)$ the sign of $\|\omega\|_h^2=\|\omega\|_g^2$ is equal to that of $(-1)^k$. Note that this choice of convention makes the right-hand-side of the last equation of \eqref{MainEq} different from how it usually appears in the literature, by a minus sign.
\end{remark}

Before we discuss  closed 4-forms $\Flu\in \Omega ^4_{\rm cl} (\X^{1, 10}) $ on $\X^{1, 10}$, let us  decompose the tangent space $\bb{V}:=T_{x}\X\simeq  \bb{R}  ^{1,10 } $  of $\X^{1, 10}$ at a point  $x\in\X$ as 
\[
\bb{V} =  \bb{L} ^{1,4}\oplus \bb{E}^6\,,
\]
where we identify $\bb{L}$ with the five-dimensional Minkowski tangent space of $\wi{M}^{1,4}$ and  $\bb{E}$ with the six-dimensional Euclidean tangent space  of $M^6$.  Then, one has an orthogonal decomposition
\[
\Lambda^{4}\bb{V}=\Lambda ^4 \bb{R}  ^{ 1,10 } =\Lambda ^4 \bb{L} \bigoplus (\Lambda ^3 \bb{L} \wedge \Lambda^{1}\bb{E}) \bigoplus (\Lambda ^2 \bb{L} \wedge \Lambda ^2 \bb{E}) \bigoplus (\Lambda^{1}\bb{L} \wedge \Lambda ^3 \bb{E})\bigoplus  \Lambda ^4 \bb{E} .
\]
In this paper, we  consider global differential 4-forms $\Flu\in\Omega^{4}(\X)$, given by
\begin{equation}
 \Flu= \varphi \tilde \al +   \tilde \be \wedge \nu +  \tilde \gamma \wedge \delta  +  \tilde \varpi  \wedge \eps  + \tilde \psi \theta \,, \label{4formdeco}
\end{equation}
for some
\begin{eqnarray*}
 &&\tilde \al \in \Omega ^4 (\wi{M}^{1,4} )\,,\quad \tilde \be \in  \Omega ^3 (\wi{ M}^{1,4} )\,,\quad \tilde  \gamma  \in \Omega  ^2 ( \wi{ M}^{1,4} )\,,\quad  \tilde \varpi \in \Omega ^1  (\wi{ M}^{1,4})\,, \quad  \tilde \psi \in C^\infty(\wi{M}^{1,4})\,,\\
&& \varphi \in C^{\infty}(M^6)\,, \quad \nu \in \Omega ^1 (M^6 )\,,\quad \delta \in \Omega ^2 (M^6)\,,\quad \eps \in \Omega ^3 (M^6)\,,\quad \theta \in \Omega ^4 (M^6)\,.
 \end{eqnarray*}
  We assume that all these differential forms are smooth and defined globally on $\widetilde{M}^{1,5}$ and $M^6$, respectively. We will show that the chosen class of  4-forms is large enough to allow for a variety of non-trivial bosonic supergravity backgrounds.  Note that the difference between (\ref{4formdeco}) and a general 4-form on $\X^{1,10}$ is, firstly, that a general 4-form may have more terms taking values in each of the subspaces $\Lambda^i \bb{L} \wedge \Lambda^{4-i} \bb{E}$ and, secondly, that each term can be multiplied by a function on $\X^{1,10}$.

\section{The closedness condition and the Maxwell equation} \label{closuremaxwell}
 We begin by writing down the closedness condition and the Maxwell equation on $(\X^{1,10}=\widetilde{M}^{1,4} \times M^6,h=\tilde g+g)$ for the 4-form $\Flu$ given by \eqref{4formdeco}.  The  closedness  condition can be found by computing $\dd \Flu$ and comparing terms of similar type, a  procedure which gives the following.
\begin{lemma} \label{CLOS}
The 4-form  $\Flu$ defined by \textnormal{(\ref{4formdeco})} is closed if and only if  the following system is satisfied:
\[
\left\{
\begin{tabular}{l l}
$\varphi \mathrm{d} \tilde \al=0$\,, & $\tilde \gamma \wedge \mathrm{d} \delta + \mathrm{d}  \tilde \varpi \wedge \eps =0$\,,\\
$\tilde \al \wedge \mathrm{d} \varphi+\mathrm{d} \tilde \beta \wedge \nu=0$\,, & $\tilde \varpi \wedge \mathrm{d} \epsilon - \mathrm{d} \tilde \psi \wedge \theta = 0$\,, \\
$\mathrm{d} \tilde  \gamma \wedge \delta -  \tilde \be \wedge \mathrm{d} \nu =0$\,, & $\tilde \psi \mathrm{d} \theta=0$\,. 
\end{tabular}\right.
\]
\end{lemma} 
 In particular, we notice that  $\Flu$ given by \eqref{4formdeco} is closed in the case that $\tilde \alpha, \tilde \beta, \tilde \gamma, \tilde \varpi, \tilde \psi,$ and $\varphi, \nu, \delta, \epsilon, \theta$ are closed on their respective manifolds. 

Before studying the Maxwell equation, we recall some basic useful formulas (see also \cite{ACT, CG}).

\begin{lemma}\label{tools}
 Let $\X=\wi M^{p}\times M^{q}$ be a product of two pseudo-Riemannian manifolds $(\tilde M, \tilde g)$ and $(M, g)$ of dimensions $p, q$, and let  $\tilde s$, $ s $ be the number of negative eigenvalues of $\tilde g $, $g$, respectively. Let us denote  by   $ \star $, $\star_{p}$, $\star_{q}$ the Hodge operator on $(\X, h=\tilde g+g) $, $ (\wi M , \tilde g )$, and $(M,  g )$, respectively. Then, for any $\tilde \al  \in \Omega ^{\tilde k} (\wi M) $,  and $ \be \in \Omega ^k (M) $ the following hold:
\[
\begin{tabular}{l l}
$\star \tilde \al =  \star_{p} \tilde \al  \wedge \mathrm{vol} _M$\,,   & $\star \vol _{M} = (-1 )^{ s}(-1) ^{pq}\vol_{\wi M}$\,, \\
$\star \vol _{\wi M}= (-1 )^{ \tilde s }\vol _M$\,, & $ \langle \tilde \al  \wedge \be , \tilde \al \wedge \be  \rangle_{h} = \langle \tilde \al , \tilde \al\rangle_{\tilde g}\langle \be , \be \rangle_{g}$\,,   \\
 $\star \be  = (-1) ^{pq} \star_{q}\be  \wedge \vol _{\wi M}$\,, & $\star (\tilde \al \wedge \be ) = (-1 )^{ k (p- \tilde k )} \star_{p} \tilde \al \wedge \star_{q} \be$\,.  
 \end{tabular}
 \]
 \end{lemma}

With the help of Lemma \ref{tools} we compute  $\star\Flu$  and  $\dd\star\Flu$ for $\Flu$ being of the form (\ref{4formdeco}). We obtain 
\begin{equation*}
\star\Flu =  \star_{5} \tilde\al \wedge \star_{6} \varphi+\star_{5} \tilde \be \wedge \star_{6} \nu +  \star_{5} \tilde \gamma \wedge \star_{6} \delta +\star_{5} \tilde \varpi \wedge \star_{6} \eps +\star_{5} \tilde \psi \wedge  \star_{6} \theta  \,,
\end{equation*}
and consequently 
\begin{eqnarray*}
\dd \star \Flu &=&
\dd\star_{5}\tilde \al \wedge \star_{6} \varphi +  \dd\star_{5} \tilde \be \wedge \star_{6} \nu + \star_{5} \tilde \be \wedge \dd \star_{6} \nu  + \dd\star_{5} \tilde \gamma \wedge\star_{6} \delta \\
 &&  -\star_{5}\tilde \gamma \wedge \dd\star_{6} \delta   +\dd\star_{5}\tilde\varpi\wedge\star_{6}\eps+  \star_{5}\tilde \varpi \wedge\dd\star_{6} \eps   - \star_{5} \tilde \psi \wedge  \dd\star_{6}\theta \,.
 \end{eqnarray*}
Notice that $\dd \star_{6} \varphi =0$ since $\star_6 \varphi$ is a 6-form on a six-dimensional manifold. Similarly, we have $\dd \star_{5} \tilde \psi =0$. 
We also compute 
\begin{eqnarray*}
\frac{1}{2} \Flu\wedge \Flu &=&  
  \varphi\,  \tilde \al \wedge  \tilde \varpi \wedge \eps  +  \varphi \, \tilde \psi \,  \tilde \al \wedge \theta + \tilde \be \wedge  \tilde \gamma \wedge \delta \wedge \nu + \tilde \be \wedge  \tilde \varpi \wedge \eps \wedge \nu  \\ 
&&+  \tilde \psi\,  \tilde \be \wedge \theta \wedge \nu  +\tilde \psi \, \tilde\gamma\wedge\theta\wedge\delta+  \tilde \gamma \wedge  \tilde \varpi \wedge \eps \wedge \delta+ \frac{1}{2}  \tilde \gamma \wedge  \tilde \gamma \wedge \delta \wedge \delta \,.
\end{eqnarray*}
After  collecting  the terms according to the subspace $\Lambda^i \mathbb{L} \wedge \Lambda^{4-i} \mathbb E$ in which they take values,  we get the following proposition. 
\begin{prop}\label{MAX}
The   Maxwell equation on the Lorentzian manifold $(\X^{1, 10} = \wi M ^{1, 4} \times M ^6, h=\tilde g +g)$ with 4-form   $\Flu$  given by \textnormal{(\ref{4formdeco})} is equivalent to the following system of equations:
\[
\begin{tabular}{ l || l l l}
Type $(\tilde 2, 6)$ & $\dd \star_{5} \tilde \al \wedge\star_{6} \varphi+ \star_{5} \tilde \be \wedge \dd\star_{6} \nu$ &$=$& $\tilde \psi \, \tilde\gamma\wedge\theta\wedge\delta$ \,, \\\\
Type $(\tilde 3, 5)$ & $\dd\star_{5}\tilde\be\wedge\star_{6}\nu-\star_{5}\tilde\gamma\wedge\dd\star_{6}\delta$ &$=$& $ \tilde \psi \, \tilde \be \wedge \theta \wedge \nu +  \tilde \gamma \wedge  \tilde \varpi \wedge \eps \wedge \delta$ \,, \\\\
Type $(\tilde 4, 4)$ & $\dd\star_{5} \tilde \gamma \wedge \star_{6} \delta  + \star_{5}\tilde\varpi\wedge\dd\star_{6}\eps$ &$=$ & 
 $\tilde \psi \, \varphi\,  \tilde  \al \wedge \theta +  \tilde \be \wedge  \tilde \varpi \wedge \eps \wedge \nu  + \frac{1}{2}  \tilde \gamma \wedge  \tilde \gamma \wedge \delta \wedge \delta$ \,, \\\\
Type $(\tilde 5, 3)$ & $\dd\star_{5} \tilde \varpi \wedge\star_{6}\eps -\star_5 \tilde \psi \wedge\dd\star_{6}\theta$  &$=$ &   $\varphi \, \tilde \al \wedge  \tilde \varpi \wedge \eps +  \tilde \be \wedge  \tilde \gamma \wedge \delta \wedge \nu$ \,.
\end{tabular}
\]
 \end{prop}

Of course, the system of equations given in Lemma \ref{CLOS} and Proposition \ref{MAX}  is significantly simplified when some terms of $\Flu$ vanish.  Let us list some of the important cases and write down the corresponding equations.  
\begin{prop}\label{casesMC}
Consider the Lorentzian manifold $(\X^{1, 10} = \wi{M}^{1, 4} \times  M ^6, h=\tilde g+g)$ as above. Then the following hold:  \\
(1) The 4-form $\Flu\in\Omega^{4}(\X)$ defined by 
\begin{equation} \label{MCalpha} 
\Flu =\varphi \tilde \al 
\end{equation} 
satisfies the Maxwell equation and the closedness  condition if and only if $\varphi$ is constant and $\tilde\al$ is closed and coclosed:
\[
\dd \tilde \al =\dd\star_{5} \tilde \al =0\,.
\]
(2) The 4-form $\Flu\in\Omega^{4}(\X)$ defined by 
\begin{equation} \label{MCbeta} 
\Flu=\tilde \be \wedge \nu  
\end{equation} 
satisfies the Maxwell equation and the closedness  condition if and only if $\tilde\be$ and $\nu$ are closed and coclosed:
\[
\dd \tilde \be  = \dd \star_{5} \tilde \be  =0\,, \qquad \dd  \nu   = \dd\star_{6}  \nu   =0\,.
\]
(3)  The 4-form $\Flu\in\Omega^{4}(\X)$ defined by 
\begin{equation} \label{MCgamma} 
\Flu =\tilde \gamma  \wedge \delta   
\end{equation} 
satisfies the Maxwell equation and the closedness  condition if and only if 
\[
\dd \tilde \gamma  = \dd\delta =\dd\star_{6}\delta  =0\,, \qquad
 \dd\star_{5} \tilde \gamma  \wedge \star_{6} \delta =\frac{\tilde \gamma \wedge \tilde \gamma \wedge \delta \wedge \delta }{2}\,.
\]
If $\tilde \gamma \wedge \tilde \gamma=0$, then the last equation implies $\dd \star_5 \tilde \gamma=0$, and it puts no additional constraints on $\delta$. If  $ \tilde \gamma \wedge \tilde \gamma$ is nonzero, the last  condition is equivalent to
\[
\dd\star_{5} \tilde \gamma = \kappa \tilde \gamma \wedge \tilde \gamma \,, \qquad \kappa
\star_{6} \delta =\frac{\delta \wedge \delta }{2}\,,
\]
for some constant $\kappa \in \mathbb R$.   \\
(4)  The 4-form $\Flu\in\Omega^{4}(\X)$ defined by 
\begin{equation} \label{MCvarpi} 
\Flu =\tilde \varpi  \wedge \eps  
\end{equation} 
satisfies the Maxwell equation and the closedness  condition if and only if  $\tilde\varpi$ and $\eps$ are closed and coclosed:
\[
\dd\tilde \varpi   = \dd \star_{5} \tilde \varpi   =0\,, \qquad \dd \eps     =\dd\star_{6} \eps=0\,.
\]
(5)  The 4-form $\Flu\in\Omega^{4}(\X)$ defined by 
\begin{equation} \label{MCtheta} 
\Flu =\tilde \psi \theta 
\end{equation} 
satisfies the Maxwell equation and the closedness  condition if and only if $\tilde \psi$ is constant and $\theta$ is closed and coclosed:
\[
\dd \theta  = \dd\star_{6}\theta  =0\,.
\]
(6)  The 4-form $\Flu\in\Omega^{4}(\X)$ defined by 
\begin{equation}\label{MCsumleft} 
\Flu=\varphi \tilde \al + \tilde \be \wedge \nu 
\end{equation} 
satisfies the Maxwell equation and the closedness  condition if and only if 
\begin{equation}\label{MCsumlefteq}
\dd\tilde  \al =\dd\star_{5}\tilde\be=\dd \nu=0\,,  \quad \dd\varphi=\kappa \nu\,, \quad \dd \tilde \beta = -\kappa \tilde \alpha\,, \quad \dd \star_5 \tilde \alpha = -\lambda \star_5 \tilde \beta\,, \quad \dd\star_6 \nu = \lambda \star_6 \varphi
\end{equation}
for some constants $\kappa, \lambda \in \mathbb R$.
The last four conditions imply 
\[\star_6 \dd \star_6 \dd \varphi = \kappa \lambda \varphi\,, \qquad \star_5 \dd \star_5 \dd \tilde \beta =  \kappa \lambda \tilde \beta \,.\]
(7)  The 4-form $\Flu\in\Omega^{4}(\X)$ defined by 
\begin{equation}\label{MCsumright} 
\Flu=\tilde \varpi \wedge \eps  + \tilde \psi \theta
\end{equation} 
satisfies the Maxwell equation and the closedness  condition if and only if 
\begin{equation}
\dd \theta =\dd \tilde \varpi =\dd\star_{6} \eps=0 \,, \quad \dd \tilde \psi=\kappa \tilde \varpi\,, \quad \dd \epsilon = \kappa \theta\,,\quad  \dd \star_5 \tilde \varpi = \lambda \star_5 \tilde \psi\,, \quad \dd\star_6 \theta = \lambda \star_6 \epsilon \label{MCsumrighteq}
\end{equation}
for some constants $\kappa,\lambda \in \mathbb R$. The last four conditions imply 
\[\star_5 \dd \star_5 \dd \tilde \psi =  \kappa \lambda  \tilde \psi\,, \qquad \star_6 \dd \star_6 \dd \epsilon = -\kappa \lambda \epsilon\,.\]
(8)  The 4-form $\Flu\in\Omega^{4}(\X)$ defined by 
\begin{equation}\label{MCalphtheta} 
\Flu =\varphi \tilde\al  + \tilde \psi \theta
\end{equation} 
satisfies the Maxwell equation and the closedness condition if and only if either $\varphi \tilde\al =0 $  or $\tilde \psi \theta =0 $. Thus, this case reduces to case (5) or case (1), respectively. \\ 
(9) The   4-form $\Flu\in\Omega^{4}(\X)$ defined by 
\begin{equation}\label{MCbnwe} 
\Flu = \tilde \be \wedge \nu + \tilde \varpi \wedge \eps 
\end{equation}
satisfies the Maxwell equation and the closedness  condition if and only if 
\[
\dd\tilde\be=\dd\nu=\dd\tilde\varpi=\dd\eps=0 \,,\quad \dd\star_{6}\nu=\dd\star_{5}\tilde\be=\dd\star_{5}\tilde\varpi =0\,, \quad \star_5 \tilde \varpi \wedge \dd \star_6 \epsilon=\tilde \beta \wedge \tilde \varpi \wedge \epsilon \wedge \nu\,.
\]
If $\epsilon \wedge \nu=0$, then the last equation implies $\dd \star_6 \epsilon=0$. If $\epsilon \wedge \nu$ is nonzero, then the last equation is equivalent to
\[ \dd\star_{6}\eps=\kappa \eps\wedge\nu\,,\qquad \kappa \star_{5}\tilde\varpi=\tilde\be\wedge\tilde\varpi\,,
\]
for some constant $\kappa \in \mathbb R$.  
\end{prop}
\begin{proof}
All the cases  are direct consequences of  Lemma \ref{CLOS} and Proposition \ref{MAX}, i.e. the closedness condition and Maxwell equation for $\Flu$ of the form \eqref{4formdeco}. We show the calculations for (6) in detail. For $\Flu=\varphi \tilde \alpha + \tilde \beta \wedge \nu$, the first three equations of Lemma \ref{CLOS} reduce to 
\[\varphi \dd \tilde \alpha=0\,, \quad \tilde \alpha \wedge \dd \varphi+\dd \tilde \beta \wedge \nu=0\,, \quad \tilde \beta \wedge \dd \nu=0\,,
\] while the last three equations hold automatically.  In particular, we see that $\dd \tilde \alpha =0$ and $\dd \nu=0$ (assuming $\varphi \neq 0$ and $\tilde \beta \neq 0$). We also have $\dd \varphi = \kappa \nu$ and $\dd \tilde \beta = -\kappa \tilde \alpha$ for some constant $\kappa \in \mathbb R$. The first two equations of Proposition \ref{MAX} reduce to 
$\dd \star_5 \tilde \alpha \wedge \star_6 \varphi+\star_5 \tilde \beta \wedge \dd \star_6 \nu = 0$ and  $\dd \star_5 \tilde \beta \wedge \star_6 \nu = 0$, respectively, while the last two hold automatically.  This implies $\dd \star_5 \tilde \beta=0$, $\dd \star_6 \nu= \lambda \star_6 \varphi$ and $\dd \star_5 \tilde \alpha = - \lambda \star_5 \tilde \beta$. All together, we obtain \eqref{MCsumlefteq}.  The other cases are treated similarly.
\end{proof}

\begin{remark}
In comparison with the examination of $(6, 5)$-decomposable supergravity backgrounds presented in \cite{CG},  i.e. Lorentzian manifolds  of the form $\Y=\wi{M}^{1, 5}\times M^5$,  we see that the system of the closedness condition and the Maxwell equation are very similar, although non-identical. In particular, a comparison of our Proposition \ref{spck} with  \cite[Prop. 2.5]{CG} shows that when the 4-form $\Flu$ is determined via  one of the  cases (1), (2), or  (4)-(8), then we obtain very similar constraints. On the other hand, the cases (3) and (9)  are quite different.  For example, in our case $\X=\wi{M}^{1, 4}\times M^6$ the Maxwell equation contains the new term $\tilde \psi \tilde\gamma\wedge\theta\wedge\delta$.  On the other hand, it does not contain the term corresponding to $\tilde\al\wedge\tilde\gamma\wedge\delta$, which appears in the Maxwell equation for  $\Y=\wi{M}^{1, 5}\times M^5$.
\end{remark}

For certain   flux forms, imposing topological restrictions  on $M^6$ may result in additional conditions on $\Flu$. For example, let $\Flu = \varphi \tilde \alpha+\tilde \beta \wedge \nu$ be a 4-form satisfying the Maxwell equation and the closedness condition. We assume that $\tilde \alpha, \tilde \beta, \nu$ are nonzero differential forms, and also that the function $\varphi$ is nonzero at every point in $M^6$.  Equation \eqref{MCsumlefteq} of Proposition \ref{casesMC} implies that $\dd \star_6 \nu = \lambda \star_6 \varphi$ or, equivalently, $\dd \star_6 \nu= \lambda \varphi \vol_M$, for some constant $\lambda \in \mathbb R$. Now, assume that $M^6$ is a closed manifold. We will see that this implies $\lambda=0$. If $\lambda \neq 0$, then there exists a constant $K>0$ such that $|\lambda \varphi| \geq K$. Thus \[ K \vol(M) = K \int_M \vol_M \leq \int_M |\lambda \varphi| \vol_M.\] The function $\lambda \varphi$ is either positive at each point in $M^6$ or negative at each point. Therefore, by Stokes' theorem we deduce that  for a connected closed manifold $M^6$ the right-hand-side is (up to an overall sign) equal to 
\[\int_M \lambda \varphi \vol_M = \int_M \dd\star_6 \nu = \int_{\partial M} \star_6 \nu =0\,.\]  
 This implies $\vol(M)=0$, a contradiction.   In particular, if $\varphi$ is a nonzero constant, we get the following statement (after absorbing the constant $\varphi$ into $\tilde \alpha$).  

\begin{prop} \label{compact}
Let $(\X^{1,10}=\widetilde{M}^{1,4} \times M^6, h=\tilde g+g, \Flu= \tilde \alpha+\tilde \beta \wedge \nu)$ be an eleven-dimensional bosonic supergravity background. If $M^6$ is closed, then $\lambda=0$ in equation \textnormal{(\ref{MCsumlefteq})}, i.e. $\dd \star_5 \tilde{\alpha} =0$ and $\dd \star_6 \nu=0$.
\end{prop} 
A similar phenomenon occurs for the flux form $\Flu = \tilde \varpi \wedge \epsilon + \theta$. 
\begin{prop} 
Let $(\X^{1,10}=\widetilde{M}^{1,4} \times M^6, h=\tilde g+g, \Flu= \tilde \varpi \wedge \epsilon+\theta)$ be an eleven-dimensional bosonic supergravity background. If $\widetilde{M}^{1,4}$ is closed, then $\lambda=0$ in equation \textnormal{\eqref{MCsumrighteq}}, i.e. $\dd \star_5 \tilde{\varpi} =0$ and $\dd \star_6 \theta=0$.
\end{prop}
\begin{proof}
By assumption, we have $\tilde \psi = 1$ in \eqref{MCsumrighteq} which implies $\dd \star_5 \tilde \varpi = \lambda \star_5 1 = \lambda \vol_{\widetilde{M}}$. We have
\[ \lambda \vol(\widetilde{M}) = \int_{\widetilde{M}} \lambda \vol_{\widetilde{M}}  = \int_{\widetilde M} \dd \star_5 \tilde \varpi= \int_{\partial \widetilde{M}} \star_5 \tilde \varpi =0\,, \]
where the last equality follows since $\widetilde{M}^{1,4}$ has been assumed to be closed. 
This implies $\lambda=0$. 
\end{proof}

Before we treat the supergravity Einstein equation, let us consider one more special case for which the Maxwell equation significantly simplifies. Namely,  assume that $\tilde \psi \theta=0$ and that $\tilde \alpha,\tilde \beta,\tilde \gamma, \tilde \varpi$ share a common factor $\tilde \omega \in \Omega^1(\widetilde{M}^{1,4})$, meaning that $\tilde \alpha=\tilde \omega \wedge \hat \alpha, \tilde \beta = \tilde \omega \wedge \hat \beta, \tilde \gamma=\tilde \omega \wedge \hat \gamma, \tilde \varpi = \hat \varpi \tilde \omega$.  Since $\tilde \omega$ is a 1-form, we have $\tilde \omega \wedge \tilde \omega =0$. Therefore the right-hand-sides in Proposition \ref{MAX} vanish. We summarize this in a proposition that we will take advantage of in Section \ref{largeexample}. 

\begin{prop} \label{MaxwellProduct}
Consider the Lorentzian manifold $(\X^{1,10}=\widetilde{M}^{1,4} \times M^6, h=\tilde g+g)$ with 4-form $\Flu = \varphi \tilde \alpha + \tilde \beta \wedge \nu + \tilde \gamma \wedge \delta + \tilde \varpi \wedge \epsilon$ 
and assume that $\tilde \alpha=\tilde \omega \wedge \hat \alpha, \tilde \beta = \tilde \omega \wedge \hat \beta, \tilde \gamma=\tilde \omega \wedge \hat \gamma, \tilde \varpi = \hat \varpi \tilde \omega$ for a non-trivial 1-form $\tilde \omega \in \Omega^1(\widetilde{M}^{1,4})$ and  $\hat \alpha \in \Omega^3(\widetilde{M}^{1,4})$, $\hat \beta \in \Omega^2(\widetilde{M}^{1,4})$,  $\hat \gamma \in \Omega^1(\widetilde{M}^{1,4}), \hat \varpi \in C^{\infty}(\widetilde{M}^{1,4})$. Then, the Maxwell equation is equivalent to the following system of equations:
\begin{align*}
\dd \star_{5} \tilde \al \wedge\star_{6} \varphi+ \star_{5} \tilde \be \wedge \dd\star_{6} \nu &= 0\,, \qquad 
\dd\star_{5}\tilde\be\wedge\star_{6}\nu-\star_{5}\tilde\gamma\wedge\dd\star_{6}\delta = 0\,,\\
\dd\star_{5} \tilde \gamma \wedge \star_{6} \delta  + \star_{5}\tilde\varpi\wedge\dd\star_{6}\eps &= 0\,, \qquad 
\dd\star_{5} \tilde \varpi \wedge\star_{6}\eps = 0\,.
\end{align*}
\end{prop}

 In particular, Proposition \ref{MaxwellProduct} shows that if $\tilde \alpha, \tilde \beta, \tilde \gamma, \tilde \varpi, \nu, \delta,\epsilon$  are  coclosed on their respective manifolds and if $\tilde \alpha, \tilde \beta, \tilde \gamma, \tilde \varpi$ share a common factor $\tilde \omega$ as above,  then $\Flu = \varphi \tilde \alpha + \tilde \beta \wedge \nu + \tilde \gamma \wedge \delta + \tilde \varpi \wedge \epsilon$ satisfies the Maxwell equation.

\section{The supergravity Einstein equation} \label{einstein}
In this section we present the supergravity Einstein equation for an oriented  Lorentzian manifold of the form $\X=\wi{M}^{1, 4}\times M^6$, endowed with the product metric $h=\tilde g+g$ and the 4-form $\Flu$ defined by  (\ref{4formdeco}).  We recall that the   supergravity Einstein equation has the form
\begin{equation}\label{EINsugra}
\Ric_{h}(X,Y)=-\frac{1}{2} \langle X\lrcorner \Flu, Y\lrcorner \Flu\rangle_{h} +\frac{1}{6}h(X,Y) \|\Flu\|^2_{h}\,,
\end{equation}
where $X, Y$ are vector fields on $\X^{1, 10}$. Note that Lemma \ref{tools} implies the following: 
\begin{lemma}
Let $\Flu$ be the 4-form on $(\X^{1, 10} = \wi{M}^{1,4}\times M^6, h=\tilde g+ g)$ defined by \textnormal{(\ref{4formdeco})}.  Then, 
\begin{equation}
\label{f2norm} 
\begin{split}
\| \Flu \|^2 _h &
= \varphi^2 \| \tilde \al \|^2 _{ \tilde g }
+\| \tilde \be  \|^2_{ \tilde g } \|\nu \|^2_g 
+ \| \tilde \gamma \|^2_{ \tilde g } \| \delta \|^2 _g 
+ \| \tilde \varpi \|^2_{ \tilde g } \|\eps\|^2 _g 
+ \tilde \psi^2\|\theta  \|^2_g \,.
\end{split} 
\end{equation} 
\end{lemma}
Since $\X^{1, 10} = \wi{M}^{1,4}\times M^6$ is a direct product of pseudo-Riemannian manifolds  we have
\begin{eqnarray*}
\Ric_{h}(X,Y)&=&\Ric_{g}(X, Y)\,,\quad\forall \ X, Y\in\Gamma(TM^6)\,,\\
\Ric_{h}(\tilde X,\tilde Y)&=&\Ric_{\tilde g}(\tilde X, \tilde Y)\,,\quad\forall \ \tilde X, \tilde Y\in\Gamma(T\wi M^{1,4})\,,\\
\Ric_{h}(X, \tilde Y)&=& 0\,,  \  \quad\quad\quad\quad\quad\forall  \   X\in\Gamma(TM^6), \  \tilde Y\in\Gamma(T\wi M^{1,4})\,.
\end{eqnarray*}
This lets us split the equation (\ref{EINsugra}) into three parts. 
By using (\ref{f2norm}) we obtain explicitly each of the parts of the supergravity Einstein equation.  In particular:

For any $X, Y  \in \Gamma(TM^6) $ we compute
\begin{equation}
\label{HHRic}
\begin{split}
\Ric_{h}(X,Y)&=\frac{\varphi^2 \| \tilde{\alpha} \|^2_{ \tilde{g}}}{6}g(X,Y)\\
&+\left(\frac{\|\nu\|^2_g}{6}g(X,Y)-\frac{1}{2}\nu(X)\nu(Y)\right) \| \tilde{\beta} \|^2_{ \tilde{g}}\\
&+\left(\frac{\|\delta\|^2_g}{6}g(X,Y)-\frac{1}{2}\langle X\lrcorner \delta,Y\lrcorner \delta\rangle_g\right) \| \tilde{\gamma} \|^2_{ \tilde{g}}\\
&+\left(\frac{\|\epsilon\|^2_g}{6}g(X,Y)-\frac{1}{2}\langle X\lrcorner \epsilon,Y\lrcorner \epsilon\rangle_g \right) \| \tilde{\varpi} \|^2_{ \tilde{g}}\\
&+\left(\frac{\|\theta\|^2_g}{6}g(X,Y)-\frac{1}{2}\langle X\lrcorner \theta,Y\lrcorner \theta \rangle_g\right) \tilde \psi^2 .
\end{split}.
\end{equation}

For any $\tilde X, \tilde Y  \in \Gamma(T\wi M^{1,4}) $ we obtain 
\begin{equation}
\label{VVRic}
\begin{split}
\Ric_{h}(\tilde{X},\tilde{Y})&=
\left(\frac{\| \tilde{\alpha} \|^2_{ \tilde{g}}}{6}\tilde{g}(\tilde{X},\tilde{Y})-\frac{1}{2}\langle X\lrcorner \tilde{\alpha},Y\lrcorner \tilde{\alpha} \rangle_{\tilde{g}}\right) \varphi^2\\
&+\left(\frac{\|\tilde{\beta}\|^2_{\tilde{g}}}{6}\tilde{g}(\tilde{X},\tilde{Y})-\frac{1}{2}\langle \tilde{X}\lrcorner \tilde{\beta},\tilde{Y}\lrcorner \tilde{\beta}\rangle_{\tilde{g}}\right) \| \nu \|^2_{ g}\\
&+\left(\frac{\|\tilde{\gamma}\|^2_{\tilde{g}}}{6}\tilde{g}(\tilde{X},\tilde{Y})-\frac{1}{2}\langle \tilde{X}\lrcorner \tilde{\gamma},\tilde{Y}\lrcorner \tilde{\gamma}\rangle_{\tilde{g}}\right) \| \delta \|^2_{ g}\\
&+\left(\frac{\|\tilde{\varpi}\|^2_{\tilde{g}}}{6}\tilde{g}(\tilde{X},\tilde{Y})-\frac{1}{2}\tilde{\varpi}(\tilde{X})\tilde{\varpi}(\tilde{Y}) \right) \| \epsilon \|^2_{ g}\\
&+\frac{\tilde \psi^2\|\theta\|^2_g}{6}\tilde{g}(\tilde{X},\tilde{Y}).
\end{split}
\end{equation}

Finally, for any $X \in \Gamma(TM^6) $ and $\tilde Y \in \Gamma(T\wi M^{1,4} )$ we get the following condition:
\begin{equation}
\label{VHRic}
\begin{split}
0=\Ric_{h}(X,\tilde{Y})&=\frac{1}{2}\Big( \varphi \nu(X) \langle \tilde{\beta}, \tilde{Y}\lrcorner \tilde{\alpha} \rangle_{\tilde{g}}
-\langle \tilde{\gamma}\wedge (X\lrcorner \delta), (\tilde{Y}\lrcorner\tilde{\beta})\wedge\nu\rangle_h
\\&\qquad +\langle \tilde{\varpi}\wedge(X\lrcorner\epsilon), (\tilde{Y}\lrcorner\tilde{\gamma})\wedge\delta\rangle_h
-\tilde \psi \tilde{\varpi}(\tilde{Y}) \langle X\lrcorner \theta, \epsilon\rangle_g
\Big).
\end{split}
\end{equation}
As a summary, we state the following:
\begin{theorem} \label{EINgeneral}
Consider the manifold $\X^{1,10}=\widetilde{M}^{1,4} \times M^6$ with the product metric $h=\tilde g + g$ where $\tilde g$ is a Lorentzian metric on $\widetilde{M}^{1,4}$ and $g$ a Riemannian metric on $M^6$, and let $\Flu$ be the 4-form defined by \textnormal{\eqref{4formdeco}}. Then the eleven-dimensional supergravity Einstein equation \eqref{EINsugra} decomposes into the equations \eqref{HHRic}, \eqref{VVRic} and \eqref{VHRic}. 
\end{theorem}

Regarding  the supergravity Einstein equation for the various special cases of $\Flu$ discussed in Proposition \ref{casesMC}, we present the following result (all equations below hold for general vector fields $X,Y$ on $M^6$ and $\tilde X, \tilde Y$ on $\widetilde{M}^{1,4}$, which for brevity we will not repeat). 
\begin{prop}\label{spck}
Consider the Lorentzian manifold $(\X^{1, 10} = \wi{M}^{1, 4} \times  M ^6, h=\tilde g+g)$. 

(1) The 4-form $\Flu\in\Omega^{4}(\X)$ defined by 
\begin{equation} \label{spcasealpha} 
\Flu =\tilde \al 
\end{equation} 
satisfies the Einstein condition if and only if the following equations hold: 
\begin{equation}
\label{Ricalpha}
\begin{split}
\Ric_{h}(X,Y)&=\frac{\| \tilde{\alpha} \|^2_{ \tilde{g}}}{6}g(X,Y)\,,\\
\Ric_{h}(\tilde{X},\tilde{Y})&=\frac{\| \tilde{\alpha} \|^2_{ \tilde{g}}}{6}\tilde{g}(\tilde{X},\tilde{Y})-\frac{1}{2}\langle X\lrcorner \tilde{\alpha},Y\lrcorner \tilde{\alpha} \rangle_{\tilde{g}}\,.
\end{split}
\end{equation}

(2) The 4-form $\Flu\in\Omega^{4}(\X)$ defined by 
\begin{equation} \label{spcasebeta} 
\Flu=\tilde \be \wedge \nu  
\end{equation} 
satisfies the Einstein condition if and only if the following equations hold:
\begin{equation}
\label{Ricbeta}
\begin{split}
\Ric_{h}(X,Y)&=\left(\frac{\|\nu\|^2_g}{6}g(X,Y)-\frac{1}{2}\nu(X)\nu(Y)\right) \| \tilde{\beta} \|^2_{ \tilde{g}}\,,\\
\Ric_{h}(\tilde{X},\tilde{Y})&=\left(\frac{\|\tilde{\beta}\|^2_{\tilde{g}}}{6}\tilde{g}(\tilde{X},\tilde{Y})-\frac{1}{2}\langle \tilde{X}\lrcorner \tilde{\beta},\tilde{Y}\lrcorner \tilde{\beta}\rangle_{\tilde{g}}\right) \| \nu \|^2_{ g}\,.
\end{split}
\end{equation}

(3)  The 4-form $\Flu\in\Omega^{4}(\X)$ defined by 
\begin{equation} \label{spcasegamma} 
\Flu =\tilde \gamma  \wedge \delta   
\end{equation} 
satisfies the Einstein condition if and only if the following equations hold:
\begin{equation}
\label{Ricgamma}
\begin{split}
\Ric_{h}(X,Y)&=\left(\frac{\|\delta\|^2_g}{6}g(X,Y)-\frac{1}{2}\langle X\lrcorner \delta, Y\lrcorner \delta\rangle_g\right) \| \tilde{\gamma} \|^2_{ \tilde{g}}\,,\\
\Ric_{h}(\tilde{X},\tilde{Y})&=\left(\frac{\|\tilde{\gamma}\|^2_{\tilde{g}}}{6}\tilde{g}(\tilde{X},\tilde{Y})-\frac{1}{2}\langle \tilde{X}\lrcorner \tilde{\gamma},\tilde{Y}\lrcorner \tilde{\gamma}\rangle_{\tilde{g}}\right) \| \delta \|^2_{ g}\,.
\end{split}
\end{equation}

(4)  The 4-form $\Flu\in\Omega^{4}(\X)$ defined by 
\begin{equation} \label{spcasevarpi} 
\Flu =\tilde \varpi  \wedge \eps  
\end{equation} 
satisfies the Einstein condition if and only if the following equations hold:
\begin{equation}
\label{Ricvarpi}
\begin{split}
\Ric_{h}(X,Y)&=\left(\frac{\|\epsilon\|^2_g}{6}g(X,Y)-\frac{1}{2}\langle X\lrcorner \epsilon, Y\lrcorner \epsilon\rangle_g\right) \| \tilde{\varpi} \|^2_{ \tilde{g}}\,,\\
\Ric_{h}(\tilde{X},\tilde{Y})&=\left(\frac{\|\tilde{\varpi}\|^2_{\tilde{g}}}{6}\tilde{g}(\tilde{X},\tilde{Y})-\frac{1}{2}\tilde \varpi(\tilde X) \tilde \varpi(\tilde Y)\right) \| \epsilon \|^2_{ g}\,.
\end{split}
\end{equation}  

(5)  The 4-form $\Flu\in\Omega^{4}(\X)$ defined by 
\begin{equation} \label{spcasetheta} 
\Flu =\theta 
\end{equation} 
satisfies the Einstein condition if and only if the following equations hold: 
\begin{equation}
\label{Rictheta}
\begin{split}
\Ric_{h}(X,Y)&=\frac{\|\theta\|^2_g}{6}g(X,Y)-\frac{1}{2}\langle X\lrcorner \theta,Y\lrcorner \theta \rangle_g\,,\\
\Ric_{h}(\tilde{X},\tilde{Y})&=\frac{\|\theta\|^2_g}{6}\tilde{g}(\tilde{X},\tilde{Y})\,.
\end{split}
\end{equation}

(6)  The 4-form $\Flu\in\Omega^{4}(\X)$ defined by 
\begin{equation}\label{spcasesumleft} 
\Flu=\varphi \, \tilde \al + \tilde \be \wedge \nu 
\end{equation} 
satisfies the Einstein condition if and only if the following equations hold:
\begin{equation}
\label{Ricspcasesumleft}
\begin{split}
\Ric_{h}(X,Y)&=\frac{\varphi^2 \| \tilde{\alpha} \|^2_{ \tilde{g}}}{6}g(X,Y)+\left(\frac{\|\nu\|^2_g}{6}g(X,Y)-\frac{1}{2}\nu(X)\nu(Y)\right) \| \tilde{\beta} \|^2_{ \tilde{g}}\,,\\
\Ric_{h}(\tilde{X},\tilde{Y})&=\left(\frac{\| \tilde{\alpha} \|^2_{ \tilde{g}}}{6}\tilde{g}(\tilde{X},\tilde{Y})-\frac{1}{2}\langle \tilde X\lrcorner \tilde{\alpha},\tilde Y\lrcorner \tilde{\alpha} \rangle_{\tilde{g}}\right) \varphi^2+\left(\frac{\|\tilde{\beta}\|^2_{\tilde{g}}}{6}\tilde{g}(\tilde{X},\tilde{Y})-\frac{1}{2}\langle \tilde{X}\lrcorner \tilde{\beta},\tilde{Y}\lrcorner \tilde{\beta}\rangle_{\tilde{g}}\right) \| \nu \|^2_{ g}\,,\\
0 &=\varphi \, \nu(X) \langle \tilde{\beta}, \tilde{Y}\lrcorner \tilde{\alpha} \rangle_{\tilde{g}}\,.
\end{split}
\end{equation}


(7)  The 4-form $\Flu\in\Omega^{4}(\X)$ defined by 
\begin{equation}\label{spcasesumright} 
\Flu=\tilde \varpi \wedge \eps  + \tilde \psi \, \theta
\end{equation} 
satisfies the Einstein condition if and only if the following equations hold:
\begin{equation}
\label{Ricspcasesumright}
\begin{split}
\Ric_{h}(X,Y)&=\left(\frac{\|\epsilon\|^2_g}{6}g(X,Y)-\frac{1}{2}\langle X\lrcorner \epsilon,Y\lrcorner \epsilon\rangle_g \right) \| \tilde{\varpi} \|^2_{ \tilde{g}}+\left(\frac{\|\theta\|^2_g}{6}g(X,Y)-\frac{1}{2}\langle X\lrcorner \theta,Y\lrcorner \theta \rangle_g\right) \tilde \psi^2\,,\\
\Ric_{h}(\tilde{X},\tilde{Y})&=\left(\frac{\|\tilde{\varpi}\|^2_{\tilde{g}}}{6}\tilde{g}(\tilde{X},\tilde{Y})-\frac{1}{2}\tilde{\varpi}(\tilde{X})\tilde{\varpi}(\tilde{Y}) \right) \| \epsilon \|^2_{ g}+\frac{\tilde \psi^2 \|\theta\|^2_g}{6}\tilde{g}(\tilde{X},\tilde{Y})\,,\\
0&= \tilde \psi \, \tilde{\varpi}(\tilde{Y}) \langle X\lrcorner \theta, \epsilon\rangle_g\,.
\end{split}
\end{equation}


(8)  The 4-form $\Flu\in\Omega^{4}(\X)$ defined by 
\begin{equation}\label{alphtheta} 
\Flu =\tilde\al  + \theta
\end{equation} 
satisfies the Einstein condition if and only if the following equations hold:
\begin{equation}
\label{Ricalphtheta}
\begin{split}
\Ric_{h}(X,Y)&=\frac{ \| \tilde{\alpha} \|^2_{ \tilde{g}}}{6}g(X,Y)+\frac{\|\theta\|^2_g}{6}g(X,Y)-\frac{1}{2}\langle X\lrcorner \theta,Y\lrcorner \theta \rangle_g\,,\\
\Ric_{h}(\tilde{X},\tilde{Y})&=\frac{\| \tilde{\alpha} \|^2_{ \tilde{g}}}{6}\tilde{g}(\tilde{X},\tilde{Y})-\frac{1}{2}\langle X\lrcorner \tilde{\alpha},Y\lrcorner \tilde{\alpha} \rangle_{\tilde{g}}+\frac{\|\theta\|^2_g}{6}\tilde{g}(\tilde{X},\tilde{Y})\,.
\end{split}
\end{equation}

(9) The   4-form $\Flu\in\Omega^{4}(\X)$ defined by 
\begin{equation}\label{bnwe} 
\Flu = \tilde \be \wedge \nu + \tilde \varpi \wedge \eps 
\end{equation}
satisfies the Einstein condition if and only if the following equations hold:
\begin{equation}
\label{Ricspcasebnwe}
\begin{split}
\Ric_{h}(X,Y)&=\left(\frac{\|\nu\|^2_g}{6}g(X,Y)-\frac{1}{2}\nu(X)\nu(Y)\right) \| \tilde{\beta} \|^2_{ \tilde{g}}+\left(\frac{\|\epsilon\|^2_g}{6}g(X,Y)-\frac{1}{2}\langle X\lrcorner \epsilon,Y\lrcorner \epsilon\rangle_g \right) \| \tilde{\varpi} \|^2_{ \tilde{g}}\,,\\
\Ric_{h}(\tilde{X},\tilde{Y})&=\left(\frac{\|\tilde{\beta}\|^2_{\tilde{g}}}{6}\tilde{g}(\tilde{X},\tilde{Y})-\frac{1}{2}\langle \tilde{X}\lrcorner \tilde{\beta},\tilde{Y}\lrcorner \tilde{\beta}\rangle_{\tilde{g}}\right) \| \nu \|^2_{ g}+\left(\frac{\|\tilde{\varpi}\|^2_{\tilde{g}}}{6}\tilde{g}(\tilde{X},\tilde{Y})-\frac{1}{2}\tilde{\varpi}(\tilde{X})\tilde{\varpi}(\tilde{Y}) \right) \| \epsilon \|^2_{ g}\,.
\end{split}
\end{equation}
\end{prop}
\begin{proof}
For each case, the equations involving the Ricci tensor follow directly from the equations \eqref{HHRic}, \eqref{VVRic} and \eqref{VHRic}. Note that we have simplified the cases (1) and (5) by setting $\varphi$ and $\tilde \psi$ equal to $1$, since we are only interested in solutions to the supergravity Einstein equation that also satisfy the Maxwell equation and the closedness condition (recall by Lemma \ref{CLOS} that  the closedness condition  implies that $\varphi$ and $\tilde \psi$ are constant in these particular cases). 
\end{proof}

 We see that the supergravity Einstein equation simplifies significantly when the form of $\Flu$ is further specified as above. In fact, the special form of many of the equations in Proposition \ref{spck} leads to some particular consequences that we will now investigate.
 
 Let $(\tilde x,x)$ denote a general point on $\widetilde M^{1,4} \times M^6$. For each equation in Proposition \ref{spck} we observe that the left-hand-side depends either on $\tilde x$ or $x$ (but not both), while the right-hand-side is either of the form $f_1(\tilde x) g_1( x)$ or $f_1(\tilde x) g_1( x)+f_2(\tilde x) g_2(x)$ for every pair of vector fields. From this we draw some conclusions about $f_i$ and $g_i$, which are functions on $\widetilde M^{1,4}$ and $M^6$, respectively.  They are based on the following simple observation.

\begin{lemma} \label{sumofproducts}
 Assume that $r(\tilde x)=f_1(\tilde x) g_1(x)$ (for every $\tilde x \in \widetilde M^{1,4}$ and every $x \in M^6$). Then either $f_1$ is identically equal to zero or $g_1$ is constant. 

Assume that $r(\tilde x)=f_1(\tilde x) g_1( x)+f_2(\tilde x) g_2( x)$ and that none of the functions $f_1$ and $f_2$ are identically equal to zero (if one of them is, then the situation is the same as above). Then either $g_1$ and $g_2$ are both constant, or $f_1(\tilde x)=C f_2(\tilde x)$ for some $C \in \mathbb R \setminus \{0\}$ and $g_2( x)=-C g_1( x)+D$ for some constant $D \in \mathbb R$. In the latter case we have $r(\tilde x)=D f_2(\tilde x)$. 
\end{lemma}
Notice that a similar statement holds after switching $x$ with $\tilde x$.  The application of these statements to the equations of Proposition \ref{spck} results in the following corollary. 
\begin{corol} \label{constantlength}
Assume that $(\X^{1,10}=\widetilde M^{1,4} \times M^6, h=\tilde g+g, \Flu)$ is a solution of the supergravity Einstein equations \eqref{EINsugra}. 

(1) If $\Flu=\tilde \alpha$, then $(M^6,g)$ is  Einstein   with Einstein constant $\|\tilde \alpha \|_{\tilde g}^2/6$. 

(2) If $\Flu=\tilde \beta \wedge \nu$, then $\| \tilde \beta \|_{\tilde g}^2$ is constant. 

(3) If $\Flu = \tilde \gamma \wedge \delta$ and $\| \tilde{\gamma} \|_{\tilde g}^2$ is not constant, then we have $g(X,Y) =\frac{3}{\|\delta\|_g^2} \langle X\lrcorner \delta, Y\lrcorner \delta\rangle_g$ for all $X,Y \in \Gamma(TM^6)$, which implies that $(M^6,g,\omega,J)$ is a Ricci-flat  almost Hermitian manifold with almost complex structure $J$ defined by $\omega(X,Y)=g(JX,Y)$, and K\"ahler form $\omega=\sqrt{\frac{3}{\|\delta\|_g^2}} \delta$. 

(4) If $\Flu= \tilde \varpi \wedge \epsilon$, then $\| \epsilon \|_g^2$ is constant and negative.

(5) If $\Flu = \theta$, then  $(\widetilde{M}^{1,4},\tilde g)$ is  Einstein  with Einstein constant $\| \theta \|_g^2/6>0$. 

(6) If $\Flu=\varphi \tilde \alpha + \tilde \beta \wedge \nu$, then $\| \tilde \beta \|_{\tilde g}^2$ and $\| \tilde \alpha \|_{\tilde g}^2$ are constant. 

(7) If $\Flu=\tilde \varpi \wedge \epsilon + \tilde \psi \theta$, then $\| \epsilon\|_g^2$ and $\|\theta\|_g^2$ are constant. 
\end{corol}
\begin{proof} Statements (1) and (5) are obvious. We prove the rest of them.

(2) When $Y=X$, the first equation of \eqref{Ricbeta} reduces to  
\[
 \Ric_{h}(X,X)=\left(\frac{\|\nu\|^2_g}{6}\|X\|_g^2-\frac{1}{2}\nu(X)^2\right) \| \tilde{\beta} \|^2_{ \tilde{g}}\,,
 \]
  which must hold for every $X \in \Gamma(TM^6)$. By Lemma \ref{sumofproducts}, either $\|\tilde \beta \|_{\tilde g}^2$ is constant, or
 \[
\left(\frac{\|\nu\|^2_g}{6}\|X\|_g^2-\frac{1}{2}\nu(X)^2\right) = 0\,.
\] 
Let $X$ be a nonzero vector field in the kernel of $\nu$, that is   $\nu(X)=0$.  Since $g$ is negative definite, the functions $\|\nu\|_g^2$ and $\|X\|_{g}^2$ are not identically zero. Therefore $\|\tilde \beta \|_{\tilde g}^2$ must be constant.

(3) If $\Flu = \tilde \gamma \wedge \delta$, we have from the first equation of \eqref{Ricgamma} that 
\[
\Ric_{h}(X,Y)=\left(\frac{\|\delta\|^2_g}{6}g(X,Y)-\frac{1}{2}\langle X\lrcorner \delta, Y\lrcorner \delta\rangle_g\right) \| \tilde{\gamma} \|^2_{ \tilde{g}}\,.
\]
 If $\| \tilde{\gamma} \|_{\tilde g}^2$ is not constant, it follows from Lemma \ref{sumofproducts} that $g(X,Y) =\frac{3}{\|\delta\|_g^2} \langle X\lrcorner \delta, Y\lrcorner \delta\rangle_g$. Consequently, $(M^6,g)$ is Ricci-flat. Since $\delta$ is a 2-form, $\|\delta\|_g^2$ is positive. Using the definition of $\omega$ we can write the condition as $g(X,Y) =\langle X\lrcorner \omega, Y\lrcorner \omega\rangle_g$. If $\omega(X,Z)=0$ for every $Z$, then $g(X,Y)=0$ for every $Y$, so non-degeneracy of $g$ implies non-degeneracy of $\omega$.  This allows us to define  an almost complex structure $J$ via $\omega(X,Y)=g(JX,Y)$, since we then get 
\[
J_i^j J_j^k= \omega_{ia} g^{aj} \omega_{jb} g^{bk} = - g^{aj} \omega_{ia} \omega_{bj} g^{bk} =-\text{Id}_i^k\,,
\] 
where   the  last equality  follows from $g(X,Y) =\langle X\lrcorner \omega, Y\lrcorner \omega\rangle_g$. 
Moreover, 
\[ 
g(JX,JY)=\omega(JX,J^2 Y) = \omega(JX,-Y)=\omega(Y,JX)=g(Y,X)=g(X,Y)
\]
and consequently  $(M^6,g,J)$ is an  almost Hermitian manifold.

(4) For $\tilde Y= \tilde X$ the second equation of \eqref{Ricvarpi} reduces to 
\[
\Ric_{h}(\tilde{X},\tilde{X})=\left(\frac{\|\tilde{\varpi}\|^2_{\tilde{g}}}{6}\tilde{g}(\tilde{X},\tilde{X})-\frac{1}{2}\tilde \varpi(\tilde X)^2\right)\| \epsilon \|_g^2\,.
\]
 If $\| \tilde \varpi \|_{\tilde g}^2 = 0$, let $\tilde X$ be a vector with the property $\tilde \varpi(\tilde X) \neq 0$. If $\| \tilde \varpi \|_{\tilde g}^2$ is different from $0$, let $\tilde X$ be the $\tilde g$-dual of the 1-form $\tilde{\varpi}$.  Then we obtain the equation  
\[
\Ric_h(\tilde X, \tilde X)=-\|\tilde \varpi \|_{\tilde g}^4 \| \epsilon \|_g^2 /3\,.
\]
 In both cases we see that $\| \epsilon \|_g^2$ is constant due to Lemma \ref{sumofproducts}, and it is negative since $g$ is negative definite on 3-forms.  

(6) To see that  $\| \tilde \alpha \|_{\tilde g}^2$ is constant, let us assume that $X$ is given by $X^j=\frac{1}{3} g^{ij} \nu_i$. Then, the first equation of \eqref{Ricspcasesumleft} reduces to
\begin{align*} 
\Ric_h(X,Y) &= \frac{\varphi^2 \| \tilde \alpha \|_{\tilde g}^2}{6} g(X,Y)+ \left( \frac{\|\nu\|^2_g}{6} g(X,Y)-\frac{1}{2} \nu(X)\nu(Y) \right) \| \tilde \beta \|_{\tilde g}^2 \\
&= \left(\frac{\varphi^2 \| \tilde \alpha \|_{\tilde g}^2}{6} + \left(\frac{ \|\nu\|^2_g}{6}-\frac{1}{2} \nu(X)\right) \| \tilde \beta \|_{\tilde g}^2 \right) \nu(Y)\\ 
&= \frac{\varphi^2 \| \tilde \alpha \|_{\tilde g}^2}{6} \nu(Y)\,,
\end{align*} 
since $\nu(X)= \| \nu \|_g^2/3$.  By choosing $Y$ such that $\nu(Y)=3g(X,Y) \neq 0$, we see that $\| \tilde \alpha \|_{\tilde g}^2$ is constant by Lemma \ref{sumofproducts}. It is then clear that also $\|\tilde \beta \|_{\tilde g}^2$ must be constant. 

(7) The proof is similar to that of (6). 
\end{proof}

\begin{remark}
The sign of the Einstein constant depends on the signature convention. For example, if the manifold $(\widetilde{M}^{1,4},\tilde g)$ in point (5)  has positive Einstein constant, it will have negative Einstein constant in the ``mostly plus'' convention. 
\end{remark}

\begin{remark}
For point (3) in Corollary \ref{constantlength} we observe the following. If $(\tilde X^{1,10},h,\Flu=\tilde \gamma \wedge \delta)$ satisfies the closedness condition, then $\dd \delta=0$. If $\|\delta\|_g^2$ is constant, this implies that $\omega$ is closed and so $(M^6,g,J,\omega)$ must be an  almost K\"ahler manifold. 
\end{remark}
 From taking the trace of the Einstein equation, it follows that the scalar curvature $\mathsf{Scal}_h$ of a bosonic supergravity background $(\X^{1,10},h, \Flu)$ satisfies the relation $\mathsf{Scal}_h=\frac{1}{6} \| \Flu \|^2_h$  (see for example \cite{FOFP, H16}). For the cases (1) and (5) in Corollary \ref{constantlength} we obtain that $\| \Flu \|^2_h$ is constant.
\begin{corol}\label{useforflags}
Assume that $(\X^{1,10}= \widetilde{M}^{1,4} \times M^6, h=\tilde g + g,\Flu)$ is a solution of the supergravity Einstein equations \eqref{EINsugra}. 

If $\Flu=\tilde \alpha$, then $\mathsf{Scal}_h=\frac{1}{6} \| \tilde \alpha \|_{\tilde g}^2$ is constant.

If $\Flu=\theta$, then $\mathsf{Scal}_h=\frac{1}{6} \|  \theta \|_{ g}^2$ is constant and positive. 
\end{corol}

\section{General theorems regarding flux forms composed of null forms} \label{null}
 We continue our investigation of manifolds of the form $\X^{1,10}= \widetilde{M}^{1,4} \times M^6$, endowed with the product metric  $h=\tilde g + g$ and the 4-form $\Flu$ given in (\ref{4formdeco}). 
Since $\tilde g$ is  a Lorentzian metric, there exist differential forms on $\widetilde{M}^{1,4}$ which are null. Recall that a $k$-form $\omega \in \Omega^k(\widetilde{M}^{1,4})$ is called {\it null} if $\langle \omega,\omega \rangle_{\tilde g}=0$. In this section we will show   that if $\Flu$ is composed of such forms, then  the supergravity Einstein equation simplifies.  These results are of particular relevance whenever $\widetilde{M}^{1,4}$ comes equipped with a distribution of null lines.  Indeed, this is the case for example if $\widetilde{M}^{1,4}$ is a {\it Walker manifold}, or a {\it Kundt spacetime}.  The case when $\widetilde{M}^{1,4}$ is a Walker manifold will be further explored in Section \ref{walker} in a way analogous to what was done in \cite{CG}.

 The following proposition concerns a particular type of null flux forms and is  a direct consequence of the relations (\ref{HHRic}) and (\ref{VVRic}).
\begin{prop} \label{GeneralNull}
Consider the Lorentzian manifold $\X^{1,10}=\widetilde{M}^{1,4} \times M^6$ with metric $h=\tilde{g}+g$ and the 4-form $\Flu$,  given by \textnormal{(\ref{4formdeco})}. Assume that $\tilde \alpha, \tilde \beta, \tilde \gamma,\tilde \varpi$ are null and moreover that $\tilde \psi=0$. Then, $(\X^{1,10},h,\Flu)$ satisfies the supergravity Einstein equations if and only if $(M^6,g)$ is Ricci-flat and 
\begin{align*}
\Ric_{h}(\tilde{X},\tilde{Y})&= -\frac{1}{2} \langle \tilde X\lrcorner \Flu, \tilde Y\lrcorner \Flu\rangle_{h}\\ 
&= -\frac{1}{2} \Big(\langle \tilde X\lrcorner \tilde{\alpha},\tilde Y\lrcorner \tilde{\alpha} \rangle_{\tilde{g}} \varphi^2+\langle \tilde{X}\lrcorner \tilde{\beta},\tilde{Y}\lrcorner \tilde{\beta}\rangle_{\tilde{g}} \| \nu \|^2_{ g}
+\langle \tilde{X}\lrcorner \tilde{\gamma},\tilde{Y}\lrcorner \tilde{\gamma}\rangle_{\tilde{g}} \| \delta \|^2_{ g}
+\tilde{\varpi}(\tilde{X})\tilde{\varpi}(\tilde{Y})\| \epsilon \|^2_{ g} \Big), \\
 0&= \varphi \nu(Z) \langle \tilde{\beta}, \tilde{Y}\lrcorner \tilde{\alpha} \rangle_{\tilde{g}}
-\langle \tilde{\gamma}\wedge (Z\lrcorner \delta), (\tilde{Y}\lrcorner\tilde{\beta})\wedge\nu\rangle_h
 +\langle \tilde{\varpi}\wedge(Z\lrcorner\epsilon), (\tilde{Y}\lrcorner\tilde{\gamma})\wedge\delta\rangle_h
\end{align*}
for any $\tilde X, \tilde Y \in \Gamma(T\widetilde{M}^{1,4})$ and any $Z \in \Gamma(TM^6)$. 
\end{prop}

One notable consequence is that  the component $(M^6,g)$ of  bosonic supergravity backgrounds of the type described in Proposition \ref{GeneralNull}  is required to be Ricci-flat. Moreover, by combining Proposition \ref{GeneralNull} with Proposition \ref{casesMC}, we  arrive at the following general statements.


\begin{theorem} \label{th:VVRicalpha}
Consider the Lorentzian manifold $\X^{1,10}=\widetilde{M}^{1,4} \times M^6$ with metric $h=\tilde{g}+g$ and the 4-form $\Flu= \tilde{\alpha}$, where $\tilde \alpha$ is null. Then $(\X^{1,10},h,\Flu)$ is a bosonic supergravity background if and only if $M^6$ is Ricci-flat, $\tilde{\alpha}$ is closed and co-closed on $M^6$, and 
\begin{equation}
\label{VVRicalpha}
\Ric_{h}(\tilde{X},\tilde{Y})=-\frac{1}{2}\langle \tilde{X}\lrcorner \tilde{\alpha},\tilde{Y}\lrcorner \tilde{\alpha} \rangle_{\tilde{g}}, \quad \forall \ \tilde{X}, \tilde{Y} \in \Gamma(T\widetilde{M}).
\end{equation}
\end{theorem}
Note that if $\Flu= \varphi \tilde{\alpha}$, then Proposition \ref{casesMC} implies that $\varphi$ is constant, and it can thus be absorbed into $\tilde{\alpha}$. Thus, above   the  condition $\Flu= \tilde{\alpha}$ is considered without loss of generality. 

\begin{theorem} \label{th:VVricbeta}
Consider the Lorentzian manifold $\X^{1,10}=\widetilde{M}^{1,4} \times M^6$ with metric $h=\tilde{g}+g$ and the 4-form $\Flu=\tilde{\beta} \wedge \nu$,  where $\tilde{\beta}$ is null. Then  $(\X^{1,10},h,\Flu)$ is a bosonic supergravity background if and only if $M^6$ is Ricci-flat, $\tilde{\beta}$ and $\nu$ are closed and co-closed (on $\widetilde{M}^{1,4}$ and $M^6$, respectively), and 
\begin{equation}
\label{VVRicbeta}
\Ric_{h}(\tilde{X},\tilde{Y})=-\frac{1}{2}\langle \tilde{X}\lrcorner \tilde{\beta},\tilde{Y}\lrcorner \tilde{\beta} \rangle_{\tilde{g}} \|\nu\|_g^2, \quad \forall \ \tilde{X}, \tilde{Y} \in \Gamma(T\widetilde{M}).
\end{equation}
\end{theorem}

\begin{theorem} \label{th:VVricGamma}
Consider the Lorentzian manifold $\X^{1,10}=\widetilde{M}^{1,4} \times M^6$ with metric $h=\tilde{g}+g$ and the 4-form $\Flu=\tilde{\gamma} \wedge \delta$,  where $\tilde{\gamma}$ is null. Then  $(\X^{1,10},h,\Flu)$ is a bosonic supergravity background if and only if $M^6$ is Ricci-flat, $\tilde{\gamma}$ and $\delta$ are closed (on $\widetilde{M}^{1,4}$ and $M^6$, respectively), $\delta$ is co-closed  on $M$ and the following two equations hold:
\begin{equation}
\label{VVRicgamma}
\dd \star_5 \tilde{\gamma} \wedge \star_6 \delta=\frac{1}{2} \tilde \gamma \wedge \tilde \gamma \wedge \delta \wedge \delta, \quad \Ric_{h}(\tilde{X},\tilde{Y})=-\frac{1}{2}\langle \tilde{X}\lrcorner \tilde{\gamma},\tilde{Y}\lrcorner \tilde{\gamma} \rangle_{\tilde{g}} \|\delta\|_g^2, \quad \forall \ \tilde{X}, \tilde{Y} \in \Gamma(T\widetilde{M}).
\end{equation}
\end{theorem}
Note that the first condition of (\ref{VVRicgamma}) is satisfied if $\tilde \gamma$ is co-closed and either $\tilde \gamma \wedge \tilde \gamma=0$ or $\delta \wedge \delta=0$.

\begin{theorem} \label{th:VVricvarpi}
Consider the Lorentzian manifold $\X^{1,10}=\widetilde{M}^{1,4} \times M^6$ with metric $h=\tilde{g}+g$ and the 4-form $\Flu=\tilde{\varpi} \wedge \epsilon$,  where $\tilde{\varpi}$ is null. Then  $(\X^{1,10},h,\Flu)$ is a bosonic supergravity background if and only if $M^6$ is Ricci-flat, $\tilde{\varpi}$ and $\epsilon$ are closed and co-closed (on $\widetilde{M}^{1,4}$ and $M^6$, respectively), and 
\begin{equation}
\label{VVRicvarpi}
\Ric_{h}(\tilde{X},\tilde{Y})=-\frac{1}{2}\tilde{\varpi}(\tilde X) \tilde{\varpi}(\tilde Y) \|\epsilon\|_g^2, \quad \forall \ \tilde{X}, \tilde{Y}  \in \Gamma(T\widetilde{M}).
\end{equation}
\end{theorem}
Observe that the equation \eqref{VVRicvarpi} implies that $\|\epsilon \|_g^2$ is constant, and  without loss of generality one may assume that it is equal to $-1$  (by absorbing the constant into $\tilde \varpi$). Let us also recall the following definition (see for example \cite{Galaev15}).
\begin{definition} \label{defRicciIsotropic}
A Lorentzian manifold $(\mathsf{X}, h)$ is called {\it totally Ricci-isotropic} if the Ricci endomorphism ${\rm ric}^{h}  : T\mathsf{X}\to T\mathsf{X}$ corresponding to $\Ric^{h}$ satisfies the relation
\[
h(\ric^{h}(X), \ric^{h}(Y))=0\,,\quad \forall \ X, Y\in\Gamma(T\X)\,.
\]
\end{definition}

In Theorem \ref{th:VVricvarpi} we observe that a bosonic supergravity background $(\X^{1,10}=\widetilde{M}^{1,4} \times M^6, h=\tilde{g}+g)$,  with flux form $\Flu=\tilde \varpi \wedge \epsilon$ and $\tilde \varpi$  null, is totally Ricci-isotropic. 
This is essentially the same claim as \cite[Cor. 4.10]{CG}, and the proof is similar.

\begin{corol}\label{cor:RicIso}
Let $\tilde \varpi \in \Omega^2(\widetilde{M}^{1,4})$ be null. Then, a bosonic supergravity background $(\X^{1,10}=\widetilde{M}^{1,4} \times M^6, h=\tilde g + g, \Flu = \tilde \varpi \wedge \epsilon)$ is totally Ricci-isotropic. 
\end{corol} 

Next, we state two results concerning flux forms of the form $\varphi \tilde \alpha + \tilde \beta \wedge \nu$ and $\tilde \beta \wedge \nu+\tilde \varpi \wedge \epsilon$, respectively, where $\tilde \alpha, \tilde \beta$ and $\tilde \varpi$ are null. 

\begin{theorem} \label{alpha-beta-nu}
Consider the Lorentzian manifold $\X^{1,10}=\widetilde{M}^{1,4} \times M^6$ with metric $h=\tilde{g}+g$ and the 4-form $\Flu=\varphi \tilde{\alpha}+ \tilde{\beta} \wedge \nu$,  where $\tilde{\alpha}$ and $\tilde{\beta}$ are null. Then  $(\X^{1,10},h,\Flu)$ is a bosonic supergravity background if and only if $M^6$ is Ricci-flat, 
\begin{equation*}
\dd\tilde  \al =\dd \nu =0\,,  \quad \dd\varphi=\kappa \nu\,, \quad \dd \tilde \beta = -\kappa \tilde \alpha\,,\quad \dd\star_{5}\tilde\be=0\,, \quad \dd \star_5 \tilde \alpha = -\lambda \star_5 \tilde \beta\,, \quad \dd\star_6 \nu = \lambda \star_6 \varphi
\end{equation*}
for some constants $\kappa, \lambda \in \mathbb R$ and
\begin{equation*}
\begin{split}
\Ric_{h}(\tilde{X},\tilde{Y})&=-\frac{1}{2}\langle \tilde X\lrcorner \tilde{\alpha},\tilde Y\lrcorner \tilde{\alpha} \rangle_{\tilde{g}} \varphi^2-\frac{1}{2}\langle \tilde{X}\lrcorner \tilde{\beta},\tilde{Y}\lrcorner \tilde{\beta}\rangle_{\tilde{g}} \| \nu \|^2_{ g}, \quad \forall \tilde{X}, \tilde{Y} \in \Gamma(T\widetilde{M}), \\
0 &= \langle \tilde{\beta}, \tilde{X}\lrcorner \tilde{\alpha} \rangle_{\tilde{g}} , \quad \forall \tilde{X} \in \Gamma(T\widetilde{M}).
\end{split}
\end{equation*}
\end{theorem}
%

\begin{theorem} \label{beta-nu-varpi-epsilon}
Consider the Lorentzian manifold $\X^{1,10}=\widetilde{M}^{1,4} \times M^6$ with metric $h=\tilde{g}+g$ and the 4-form $\Flu=\tilde{\beta} \wedge \nu + \tilde{\varpi} \wedge \epsilon$,  where $\tilde{\beta}$ and $\tilde{\varpi}$ are null. Then  $(\X^{1,10},h,\Flu)$ is a bosonic supergravity background if and only if $M^6$ is Ricci-flat, 
\[
\dd\tilde\be=\dd\nu=\dd\tilde\varpi=\dd\eps=0 \,,\quad \dd\star_{6}\nu=\dd\star_{5}\tilde\be=\dd\star_{5}\tilde\varpi =0\,, \quad \star_5 \tilde \varpi \wedge \dd \star_6 \epsilon=\tilde \beta \wedge \tilde \varpi \wedge \epsilon \wedge \nu\, ,
\] and
\begin{equation*}
\Ric_{h}(\tilde{X},\tilde{Y})=-\frac{1}{2}\langle \tilde{X}\lrcorner \tilde{\beta},\tilde{Y}\lrcorner \tilde{\beta}\rangle_{\tilde{g}} \| \nu \|^2_{ g}-\frac{1}{2}\tilde{\varpi}(\tilde{X})\tilde{\varpi}(\tilde{Y}) \| \epsilon \|^2_{ g}, \quad \forall \tilde{X}, \tilde{Y} \in \Gamma(T\widetilde{M}).
\end{equation*}
\end{theorem}
 Notice that the last equation in Proposition \ref{GeneralNull}, the one coming from equation \eqref{VHRic}, is satisfied automatically in Theorems \ref{th:VVRicalpha} - \ref{beta-nu-varpi-epsilon}. The only exception is Theorem \ref{alpha-beta-nu} where the consequence $\langle \tilde{\beta}, \tilde{X}\lrcorner \tilde{\alpha} \rangle_{\tilde{g}}=0$ gives additional constraints.  
 
 Let us also pose the following:
\begin{corol}
 All bosonic supergravity backgrounds appearing in this section have vanishing scalar curvature.
\end{corol}
\begin{proof}This is a simple consequence of the relation 
  $\mathsf{Scal}_h=\frac{1}{6} \| \Flu \|^2_h$ and the fact that the  flux form $\Flu$ is null for all the backgrounds presented above.
\end{proof}

\section{Bosonic supergravity backgrounds for which $\widetilde{M}^{1,4}$ is a Ricci-isotropic Walker manifold} \label{walker}
In order to construct explicit examples of bosonic supergravity backgrounds for which the flux form is composed of null forms, as treated in the previous section, we assume in this section that $\widetilde{M}^{1,4}$ is a {\it Lorentzian  Walker manifold}. Lorentzian Walker manifolds admit a parallel distribution of isotropic lines which we will use to build the Lorentzian part of the flux form $\Flu$. 
 With the aim to further simplify the supergravity Einstein equations, we follow \cite {CG} and will work with special type of totally Ricci-isotropic Walker manifolds, defined below. In Sections \ref{sect:alpha} - \ref{sectvarpi} we consider the simplest type of flux forms, namely those of the form  $\tilde \alpha$, $\tilde \beta \wedge \nu$, $\tilde \gamma \wedge \delta$, and  $\tilde \varpi \wedge \epsilon$. In Section \ref{largeexample} we unify these results by considering the more general flux form  $\varphi \tilde \alpha+\tilde \beta \wedge \nu+\tilde \gamma \wedge \delta+\tilde \varpi \wedge \epsilon$, 
 under the additional condition that the eight involved differential forms $\tilde \alpha,\tilde \beta, \tilde \gamma, \tilde \varpi, \varphi, \nu,\delta,\epsilon$ are closed and coclosed on their respective manifolds. Finally, in Section \ref{sectalphabeta} we consider the flux form $\varphi \tilde \alpha + \tilde \beta \wedge \nu$, without the strict assumption of closedness and coclosedness on each of its components.

\subsection{Ricci-isotropic Walker manifolds and null forms} \label{Ricci-isotropic}
Let us recall  the definition of a Lorentzian Walker manifold. 
\begin{definition}
A {\it Lorentzian Walker manifold} is a Lorentzian manifold that admits a parallel distribution of isotropic (or null) lines.
\end{definition}
 Next we focus on the five-dimensional case.  If $(\widetilde{M}^{1,4},\tilde g)$ is a  Lorentzian Walker manifold, then it is locally diffeomorphic to a product $\mathbb R \times N^3 \times \mathbb R$ of manifolds with coordinates $u$, $x=(x^1,x^2,x^3)$ and $v$, respectively, on which the metric takes the form 
\begin{equation}
\tilde g = 2 \dd u \dd v +\rho+2 A \dd u+H \dd u^2\,. \label{gWalker}
\end{equation}
Here, $\rho=\rho_{ij}(u,x) \dd x^i \dd x^j$ is a family of Riemannian metrics on $N^3$ (parametrized by $u$ and of signature $(-,-,-)$), $A=A_i(u,x) \dd x^i$ is a family of 1-forms on $N^3$, and $H=H(u,x,v)$ is a smooth function on $\widetilde{M}^{1,4}$ (see \cite{Wa50, G10, GL10, CG}). 
In these coordinates,  the distribution spanned by $\partial_v$ consists of isotropic lines and it is parallel, since $\nabla^{\tilde g} \partial_v=\frac{1}{2} H_v \partial_v \otimes \dd u $, where $\nabla^{\tilde g}$ denotes the Levi-Civita connection with respect to $\tilde g$. Observe also that $\dd u=\partial_v \lrcorner \tilde g$ is null, that is, $\langle\dd u,\dd u\rangle_{\tilde g}= 0$. 

Following \cite{CG}, we will further assume the following:
 \begin{equation}
 \partial_v H=0\,, \quad A_i=0 \  \text{for any} \  i=1,2,3\,,\quad    \rho \ \text{is a family of Ricci-flat metrics}\,. \label{cond:RicciIsotropic}
 \end{equation}
  The first condition implies $\nabla^{\tilde g} \partial_v=0$ and consequently $\nabla^{\tilde g} \dd u=0$. Under these assumptions, the Ricci tensor is significantly simplified to
\begin{equation}
\Ric_{\tilde g} = -\frac{1}{2} \Delta_\rho (H)\, \dd u^2\,, \label{RicciWalker}
\end{equation}
where  $\Delta_\rho (H)=\sum_{i, j,k=1}^{3} \rho^{ij} \big(\partial_{x^i} \partial_{x^j} H - \Gamma_{ij}^k \partial_{x^k} H\big)$ is the Laplace-Beltrami operator of the metric $\rho$ applied to $H$ (see \cite{G10}). It follows that the Ricci endomorphism $\mathrm{ric}_{\tilde g}$ is null, i.e. $\langle \mathrm{ric}_{\tilde g},\mathrm{ric}_{\tilde g}\rangle_{\tilde g}=0$.  As in \cite{CG},  we shall slightly abuse the terminology and call Walker metrics satisfying the conditions \eqref{cond:RicciIsotropic}   {\it Ricci-isotropic Walker metrics}, referring to the property that the image of the Ricci endomorphism  related to the Walker metric is totally null (see Definition \ref{defRicciIsotropic}).   Note  that since the 1-form $\dd u$ is null, we can use it to build other null differential forms on $\widetilde{M}^{1,4}$. In particular, if we use $\dd u$ to construct a null 4-form $\Flu$, we may take advantage of the results found in the previous section. 

By Proposition \ref{GeneralNull} we know that if $(X^{1,10}=\widetilde{M}^{1,4}\times M^6,h=\tilde g+g,\Flu)$ is a bosonic supergravity background, then the Ricci tensor of $\widetilde{M}^{1,4}$ satisfies the equation
\[\Ric_{h}(\tilde{X},\tilde{Y})= -\frac{1}{2} \Big(\langle \tilde X\lrcorner \tilde{\alpha},\tilde Y\lrcorner \tilde{\alpha} \rangle_{\tilde{g}} \varphi^2+\langle \tilde{X}\lrcorner \tilde{\beta},\tilde{Y}\lrcorner \tilde{\beta}\rangle_{\tilde{g}} \| \nu \|^2_{ g}
+\langle \tilde{X}\lrcorner \tilde{\gamma},\tilde{Y}\lrcorner \tilde{\gamma}\rangle_{\tilde{g}} \| \delta \|^2_{ g}
+\tilde{\varpi}(\tilde{X})\tilde{\varpi}(\tilde{Y})\| \epsilon \|^2_{ g} \Big)\,.
\]
When  $(\widetilde{M}^{1,4},\tilde g)$ is a  Ricci-isotropic Walker manifold of the type described above, the right-hand-side of this equation must be the same type of tensor as in \eqref{RicciWalker}.  The following lemma shows that this happens when $\tilde \alpha, \tilde \beta, \tilde \gamma, \tilde \varpi$ are of the form $\dd u \wedge \omega(u)$, for some $\omega(u) \in \Omega^k(N^3)$. Here,  the  notation indicates that the differential forms on $N^3$ are parametrized by $u$. We remind that the metric $\rho$ in general also depends on $u$ even though our notation does not emphasize this. 
\begin{lemma} \label{du2vanish}
Let $\omega(u) \in \Omega^k(N^3)$ be a $k$-form and let  $\tilde{g} = 2 \dd u \dd v+\rho+H(u,x) \dd u^2$ be a metric on $\widetilde{M}^{1,4} = \mathbb R \times N^3 \times \mathbb R$. Then 
\[\langle \tilde X\lrcorner (\dd u \wedge \omega(u)),\tilde Y\lrcorner ( \dd u \wedge \omega(u))  \rangle_{\tilde{g}}=a_1 a_2 \langle \omega(u), \omega(u) \rangle_{\tilde g}\,,
\]
where $a_1=\tilde X \lrcorner \dd u$ and $a_2=\tilde Y \lrcorner \dd u$. In particular, the expression vanishes, unless both $\tilde X \lrcorner \dd u$ and $\tilde Y \lrcorner \dd u$ are nonzero.
\end{lemma}
\begin{proof}
We have $\langle \dd u,\dd u\rangle_{\tilde g}=0$ and $  \langle \dd u,  \dd x^i\rangle_{\tilde g}=0$, 
which implies $\langle \dd u \wedge \omega_1, \omega_2  \rangle_{\tilde g}=0$
for every $k$-form $\omega_2$ and every $(k-1)$-form $\omega_1$ on $\mathbb R\times N^3$. 
Let $\tilde X= a_1 \partial_u+\sum_{i=1}^3 b^i_1 \partial_{x^i}+c_1 \partial_v$ and  $\tilde Y= a_2 \partial_u+\sum_{i=1}^3 b^i_2 \partial_{x^i}+c_2 \partial_v$.  Then, for $\tilde \omega=\dd u \wedge \omega(u)$ we compute
\begin{align*}
\langle \tilde X\lrcorner \tilde \omega,\tilde Y\lrcorner \tilde \omega  \rangle_{\tilde{g}} &= \langle a_1 \omega(u) - \dd u \wedge (\sum_{i=1}^3 b^i_1 \partial_{x^i} \lrcorner \omega(u) ), a_2 \omega(u) - \dd u \wedge (\sum_{i=1}^3 b^i_2 \partial_{x^i} \lrcorner \omega(u) ) \rangle_{\tilde g} \\&=a_1 a_2 \langle \omega(u), \omega(u) \rangle_{\tilde g}\,.
\end{align*}
\end{proof}

An important subclass of Ricci-isotropic  Lorentzian Walker metrics,  which we may use in our study to construct explicit examples of bosonic supergravity backgrounds, consists of the so-called   {\it pp-waves} (see \cite{F00, Leistner} for details). Locally, in five dimensions such manifolds are of the form (\ref{gWalker}), with $A=0$, $\rho =- (\dd x^1)^2-(\dd x^2)^2-(\dd x^3)^2$ and $\partial_v H=0$, and so topologically  $\widetilde{M}^{1, 4}=\R\times\R^3\times\R\cong\R^5$. 
In particular, we have $\Delta_\rho H=- \sum_{i=1}^3 H_{x^i x^i}$.
\begin{remark}
Walker manifolds  provide  examples of  indefinite metrics that exhibit various geometric aspects (see for example \cite{Bryant, Gib, GL10} for the Lorentzian version of such manifolds).  For  instance, the pp-waves  form one of the simplest and well-known classes of Lorentzian Walker manifolds. On the other hand,  {\it (totally) Ricci-isotropic} Lorentzian manifolds   are known to be important in holonomy theory of indefinite metrics (see for example \cite{Galaev15}), and their Ricci tensor attains a   simplified expression (\cite{G10}). Due to this special holonomy feature,  Ricci-isotropic Lorentzian Walker manifolds have many natural applications in supergravity theories, see for instance \cite{Bran, Col, Gib2,  Bryant, F00, CG}.
\end{remark} 

 In the remainder of this section, we apply the results from the previous sections to the case where $(\widetilde{M}^{1,4}, \tilde g)$ is a  Lorentzian Walker manifold satisfying \eqref{cond:RicciIsotropic}.

\subsection{Results concerning the flux form $\Flu=\tilde{\alpha}$} \label{sect:alpha}

Let us consider the non-trivial 4-form $\Flu=\tilde{\alpha}= \dd u \wedge f(u,x) \vol_\rho$, where $\vol_\rho$ denotes the (in general $u$-dependent) volume form of the metric $\rho$ on $N^3$. 
\begin{prop} \label{propalpha}
Let $(M^{6},g)$ be a Ricci-flat Riemannian manifold and let $(\widetilde{M}^{1,4},\tilde g=2 \dd u  \dd v+\rho+H \dd u^2)$ be a Walker manifold with $\rho$ Ricci-flat and $\partial_v(H)=0$. Then $(\X^{1,10}= \widetilde{M}^{1,4} \times M^6,h=\tilde g+g, \Flu=f(x,u) \dd u \wedge \vol_\rho)$ is a bosonic supergravity background if and only if $\partial_{x^i} (f)=0$ for $i=1,2,3$ and 
\[ \Delta_\rho H=-f^2.\]
\end{prop}
\begin{proof}
Since $\tilde{\alpha}$ does not depend on $v$, we have $\dd \tilde{\alpha} = 0$. The condition $\dd \star_5 \tilde{\alpha}=0$ is equivalent to $\partial_{x^i}(f)=0$ for $i=1,2,3$. From Lemma \ref{du2vanish} we see that $\langle \tilde{X}\lrcorner \tilde{\alpha},\tilde{Y}\lrcorner \tilde{\alpha} \rangle_{\tilde{g}}=0$ for every pair $\tilde{X}$, $\tilde{Y}$ on which $\dd u^2$ vanishes. The statement then follows from Theorem \ref{th:VVRicalpha} since we get 
\[
\Ric_{\tilde g} = \frac{1}{2} f^2 \dd u^2\,,
\]
 which by \eqref{RicciWalker} is equivalent to $\Delta_\rho H=-f^2$. 
\end{proof}

\begin{example} \label{Alpha}
For an explicit example, let $(M^6,g)$ be a Ricci-flat Riemannian  manifold and let $(\widetilde{M}^{1,4},\tilde g=2 \dd u \dd v-\sum_{i=1}^3 (\dd x^i)^2+H \dd u^2)$ be a five-dimensional pp-wave. Since $ \Delta_\rho H=- \sum_{i=1}^3 H_{x^i x^i}$, the equation $\Delta_\rho H=-f^2$ is satisfied when $H=\frac{1}{6} f(u)^2 \sum_{i=1}^3 (x^i)^2$. Thus, with this choice of $H$,
\[(\X^{1,10}=\widetilde{M}^{1,4} \times M^6, h=\tilde g+g, \Flu= f(u) \dd u \wedge \dd x^1 \wedge \dd x^2 \wedge \dd x^3)\]
is an eleven-dimensional bosonic supergravity background. 
\end{example}

\subsection{Results concerning the flux form $\Flu=\tilde{\beta}\wedge\nu$} \label{sectbeta}
Let us now consider a   Ricci-flat Riemannian manifold $M^6=\mathbb R \times \Sigma$ with metric $g=-\dd t^2-\mu$, where $\mu$ is a positive definite metric on the five-dimensional manifold $\Sigma$. 

\begin{prop}\label{propbeta}
Let $(M^6=\mathbb R \times \Sigma,g)$ be a  Ricci-flat Riemannian manifold with metric $g=-\dd t^2-\mu$ and let $(\widetilde{M}^{1,4}=\mathbb R \times N^3 \times \mathbb R,\tilde g=2 \dd v \dd u+\rho+H \dd u^2)$ be a Walker manifold with  $\partial_v(H)=0$ and $\rho$ $u$-independent and Ricci-flat.  Set $\nu=\dd t$ and $\tilde{\beta}=\dd u\wedge \omega$ for a closed and coclosed 2-form $\omega$ on $N^3$. Then $(\X^{1,10}=\widetilde{M}^{1,4}\times M^6,h=\tilde{g}+g,\Flu=\tilde{\beta}\wedge \nu)$ is a bosonic supergravity background if and only if 
$$
\Delta_\rho H=-\|\omega\|^2_{ \rho}.
$$
\end{prop}
\begin{proof}
It is clear that $\tilde \beta$ is closed, and that $\nu$ is closed and coclosed. It follows from  $\star_5 \tilde \beta= \star_2 \dd u \wedge \star_\rho \omega=-\dd u \wedge \star_\rho \omega$ that $\tilde \beta$ is coclosed. 
We also see that $\langle \tilde{X}\lrcorner \tilde{\beta},\tilde{Y}\lrcorner \tilde{\beta} \rangle_{\tilde{g}}=0$ for every pair $\tilde X, \tilde Y$ on which $\dd u^2$ vanishes. Thus it follows from Theorem\ \ref{th:VVricbeta} that the Ricci tensor is given by 
\[
\Ric_h = -\frac{1}{2} \|\omega\|_{\tilde g}^2 \| \nu \|_g^2 \dd u^2\,.
\]
This equivalent to $\Delta_\rho H=-\|\omega\|^2_{\tilde g}=-\|\omega\|^2_{\rho}$, since $\|\nu\|_g^2=-1$.
\end{proof}

\begin{example}
\label{Beta and Nu}
Let $(\widetilde{M}^{1,4}=\mathbb{R}\times N^3\times \mathbb{R},\tilde{g}=2 \dd v \dd u-\sum_{i=1}^3 (\dd x^i)^2+H \dd u^2)$ be a pp-wave with $H=\frac{1}{6} \sum_{i=1}^3 (x^i)^2$. Let $(M^6,g)$ and $\nu$ be as described in the proposition above and set $\tilde{\beta}= \dd u\wedge \dd x^1\wedge \dd x^2$. Then 
$$
(\X^{1,10}=\widetilde{M}^{1,4}\times M^6,h=\tilde{g}+g,\Flu=\tilde{\beta}\wedge \nu)
$$
is an eleven-dimensional bosonic supergravity background.
\end{example}

\subsection{Results concerning the flux form $\Flu=\tilde{\gamma}\wedge\delta$} \label{sectgamma}

Let $\tilde \gamma=\dd u \wedge \zeta$ for a 1-form $\zeta$ on $N^3$ and assume that $M^6$ is a Calabi-Yau manifold and that $\delta$ is its K\"ahler form.

\begin{prop} \label{propgamma}
Consider a six-dimensional Calabi-Yau manifold $(M^6,g, \delta)$, where $\delta$ denotes the K\"ahler form, and a five-dimensional Walker manifold $(\widetilde{M}^{1,4}=\mathbb{R}\times N^3\times \mathbb{R},\tilde{g}=2\dd v\dd u+\rho+H\dd u^2)$ with $\rho$ $u$-independent and Ricci-flat, and $\partial_v(H)=0$.  Set $\tilde{\gamma}=\dd u\wedge \zeta$ for some closed and coclosed $1$-form $\zeta$ on $N^3$.  Then $(\X^{1,10}=\widetilde{M}^{1,4}\times M, h=\tilde{g}+g, \Flu=\tilde{\gamma}\wedge \delta)$ is a bosonic supergravity background if and only if
\begin{equation} \label{DeltaHGammaDelta}
\Delta_\rho H=\|\zeta\|^2_{\rho}\|\delta\|^2_g.
\end{equation}
\end{prop}
\begin{proof}
We use Theorem \ref{th:VVricGamma}. Since $\tilde \gamma \wedge \tilde \gamma =0$, the Maxwell equation and closedness condition are satisfied when $\tilde \gamma$ and $\delta$ are closed and coclosed. Since $\delta$ is the K\"ahler form, it is closed and coclosed, and the same holds for $\tilde\gamma$. Since $(M^6,g,\delta)$ is a Calabi-Yau manifold, it is Ricci-flat. We also see that $\langle \tilde{X}\lrcorner \tilde{\gamma},\tilde{Y}\lrcorner \tilde{\gamma} \rangle_{\tilde{g}}=0$ for every pair $\tilde X, \tilde Y$ on which $\dd u^2$ vanishes. It follows from (\ref{VVRicgamma}) that 
\[
\Ric_{h}
=-\frac{1}{2}\|\zeta\|^2_{\tilde{g}} \|\delta\|^2_g \dd u^2\,,
\]
 which is equivalent to \eqref{DeltaHGammaDelta}  (note that the function $\|\delta\|_g^2$ is constant  since it is the K\"ahler form).  This proves our claim.
\end{proof}


\subsection{Results concerning the flux form $\Flu=\tilde{\varpi}\wedge\epsilon$} \label{sectvarpi}

\begin{prop} \label{propvarpi}
Let $(M^6,g)$ be a Riemannian Ricci-flat manifold and $\epsilon$ a closed and coclosed $3$-form on $M^6$. Let $(\widetilde{M}^{1,4}=\mathbb{R}\times N^3\times \mathbb{R},\tilde{g}=2\dd v\dd u+\rho+H\dd u^2)$ be a Walker manifold with $\rho$ $u$-independent and Ricci-flat and $\partial_v(H)=0$. Set $\tilde{\varpi}=\dd u$. Then $(\X^{1,10}=\widetilde{M}^{1,4}\times M^6,h=\tilde{g}+g,\Flu=\tilde{\varpi}\wedge \epsilon)$ is a bosonic supergravity background if and only if $\|\epsilon\|_{g}^2$ is constant and 
$$
\Delta H=\| \epsilon \|^2_{ g}.
$$
\end{prop}
\begin{proof}
We use Theorem \ref{th:VVricvarpi}. It is clear that $\tilde{\varpi}=\dd u$ is both closed and coclosed.  We have $\tilde{\varpi}(\tilde X) \tilde{\varpi}(\tilde Y)=0$ for every pair $\tilde X, \tilde Y$ on which $\dd u^2$ vanishes. Thus, by  \eqref{VVRicvarpi} the Ricci tensor takes the form 
\[
\Ric_h = -\frac{1}{2} \|\epsilon\|_g^2 \dd u^2\,,
\]
 or, equivalently, $\Delta_\rho H=\| \epsilon \|^2_{ g}$. Since the left-hand-side of this equation is a function on $\widetilde{M}^{1,4}$, $\|\epsilon\|_g^2$ must be constant. This completes the proof.
\end{proof}
We illustrate Proposition \ref{propvarpi} with the following example:
\begin{example}
Let $\tilde{M}^{1,4}$ be a pp-wave with metric 
\[
\tilde g=2 \dd v \dd u -\sum_{i=1}^3 (\dd x^i)^2-\frac{E^2}{6}\left(  \sum_{i=1}^3 (x^i)^2\right) \dd u^2\,,
\]
 let $(M^6,g)$ be a Riemannian Ricci-flat manifold and let $\epsilon$ be a closed and coclosed 3-form on $M^6$ with $\| \epsilon \|_g^2=-E^2$ constant. Then $(\X^{1,10} = \widetilde{M}^{1,4} \times M^6, h=\tilde g + g, \Flu=  \dd u \wedge \epsilon)$ is an eleven-dimensional bosonic supergravity background. 
\end{example}

\subsection{Results concerning the flux form $\Flu=  \varphi \tilde \alpha + \tilde \beta \wedge \nu + \tilde \gamma \wedge \delta+\tilde \varpi \wedge \epsilon$}  \label{largeexample}
In this section, we unify the previous four cases by considering a flux form $\Flu=  \varphi \tilde \alpha + \tilde \beta \wedge \nu + \tilde \gamma \wedge \delta+\tilde \varpi \wedge \epsilon$ where 
\[ 
\tilde \alpha= \dd u \wedge  \hat \alpha(u)\,, \quad \tilde \beta = \dd u \wedge \hat \beta(u)\,, \quad \tilde \gamma = \dd u \wedge \hat \gamma(u)\,, \quad \tilde \varpi = \dd u \wedge \hat \varpi(u)\,,
\] 
with $\hat \alpha(u) \in \Omega^3(N^3)$,  $\hat \beta(u) \in \Omega^2(N^3)$,  $\hat \gamma(u) \in \Omega^1(N^3)$, and $\hat \varpi(u) \in C^{\infty}(N^3)$, respectively.  Recall that  the  notation indicates that the differential forms on $N^3$ are parametrized by $u$. Since $\dim N^3=3$, we have $\hat \alpha(u)= f(u,x) \vol_\rho$ for some function $f \in C^{\infty}(\mathbb R \times N^3)$, where $\vol_\rho$ is the (in general $u$-dependent) volume form with respect to the metric $\rho$. 
In this case, the Maxwell equations  (Proposition \ref{MAX}) simplify significantly: all right-hand-sides vanish due to $\dd u \wedge \dd u = 0$, and we obtain the equations in Proposition \ref{MaxwellProduct}. 
It is easily seen that both the closedness condition and the Maxwell equation are satisfied  in the particular case that $\tilde \alpha, \tilde \beta, \tilde \gamma,\tilde \varpi$ and $\varphi, \nu, \delta,\epsilon$ are closed and coclosed on their respective manifolds. Let us remark that the closedness of $\varphi$ implies  its constancy, and we can without loss of generality assume that it is equal to $1$.  Finally, notice that if $\omega= \dd u \wedge \hat \omega(u)$ for some $\hat \omega(u) \in \Omega^k(N^3)$, then closedness of $\omega$ with respect to the exterior derivative on $\widetilde{M}^{1,4}$ is equivalent to closedness of $\hat \omega(u)$ with respect to the exterior derivative $\dd_N$ on $N^3$:
\[ 
\dd \omega = -\dd u \wedge \dd \hat \omega(u) = -\dd u \wedge \dd_N \hat \omega(u)\,.
\]
A similar statement can be made for coclosedness of $\omega$:   $\dd \star_5  \omega=0$ if and only if $\dd_N \star_\rho \hat \omega(u)=0$. 


\begin{prop}\label{gen1}
Consider the 4-form $\Flu= \dd u \wedge (\hat \alpha(u)+\hat \beta(u) \wedge \nu+\hat\gamma(u) \wedge \delta+\hat \varpi(u)  \epsilon)$, where 
\begin{gather*}
\hat \alpha(u) \in \Omega^3(N^3)\,, \;\; \hat \beta(u) \in \Omega^2(N^3)\,, \;\; \hat \gamma(u) \in \Omega^1(N^3)\,,  \;\; \hat \varpi(u) \in C^{\infty}(N^3)\,, \\  \nu \in \Omega^1(M^6)\,, \;\; \delta \in \Omega^2(M^6)\,, \;\; \epsilon \in \Omega^3(M^6)\,,
\end{gather*} are closed and coclosed on $N^3$ and $M^6$, respectively. Let also $(\widetilde{M}^{1,4},\tilde g)$ be a Walker metric of the form \eqref{gWalker} with $\rho$ Ricci-flat, $A=0$, and $\partial_v(H)=0$. Then,
\[ (\X^{1,10} = \widetilde{M}^{1,4} \times M^6, h=g +\tilde g, \Flu)\] 
is a bosonic supergravity background if and only if $(M^6,g)$ is Ricci-flat and 
\[\Delta_\rho H = \| \hat \alpha(u) \|_\rho^2+\| \hat \beta(u) \|_\rho^2 \| \nu \|_g^2+\|\hat \gamma(u)\|^2_\rho \| \delta \|^2_g + \hat \varpi(u)^2 \| \epsilon \|_g^2.\] 
\end{prop}

\begin{proof}
It is clear that the closedness condition in Lemma \ref{CLOS} is satisfied, and so is the Maxwell equation (Proposition \ref{MaxwellProduct}). Thus, the condition for being a bosonic supergravity background boils down to the supergravity Einstein equation, which by Proposition \ref{GeneralNull} consists of the following system of equations:
\[
\left\{\begin{split}
\Ric_{h}(\tilde{X},\tilde{Y})&=  -\frac{1}{2} \Big(\langle \tilde X\lrcorner \tilde{\alpha},\tilde Y\lrcorner \tilde{\alpha} \rangle_{\tilde{g}}+\langle \tilde{X}\lrcorner \tilde{\beta},\tilde{Y}\lrcorner \tilde{\beta}\rangle_{\tilde{g}} \| \nu \|^2_{ g}
+\langle \tilde{X}\lrcorner \tilde{\gamma},\tilde{Y}\lrcorner \tilde{\gamma}\rangle_{\tilde{g}} \| \delta \|^2_{ g}
+\tilde{\varpi}(\tilde{X})\tilde{\varpi}(\tilde{Y})\| \epsilon \|^2_{ g} \Big), \\
 0&= \varphi \nu(Z) \langle \tilde{\beta}, \tilde{Y}\lrcorner \tilde{\alpha} \rangle_{\tilde{g}}
-\langle \tilde{\gamma}\wedge (Z\lrcorner \delta), (\tilde{Y}\lrcorner\tilde{\beta})\wedge\nu\rangle_h
 +\langle \tilde{\varpi}\wedge(Z\lrcorner\epsilon), (\tilde{Y}\lrcorner\tilde{\gamma})\wedge\delta\rangle_h\,.
 \end{split}\right.
\]
The first equation reduces to 
\[ \Delta_\rho H = -2\Ric_h( \partial_u,\partial_u) =     \| \hat \alpha(u) \|_\rho^2+\| \hat \beta(u) \|_\rho^2 \| \nu \|_g^2+\|\hat \gamma(u)\|^2_\rho \| \delta \|^2_g + \hat \varpi(u)^2 \| \epsilon \|_g^2\] 
due to \eqref{RicciWalker} and Lemma \ref{du2vanish}, while the second one holds automatically. Thus we get our claim.

\end{proof}

Notice that since $\hat \alpha(u)=f(u,x) \vol_\rho$, we have $\dd_N \hat \alpha(u)=0$. We also see that $\| \hat \alpha(u) \|_\rho^2=- f^2$, and the condition $\dd_N \star_\rho \hat \alpha(u)=0$ is equivalent to condition $\partial_{x^i}(f)=0$ for $i=1,2,3$, which we recognize from Proposition \ref{propalpha}. Notice that closedness of $\hat \varpi(u)$ on $N^3$ means that $\hat \varpi(u)$ is a function of $u$ only.  
 The propositions \ref{propalpha}, \ref{propbeta}, \ref{propgamma}, \ref{propvarpi} can now be viewed as corollaries of Proposition \ref{gen1}. Moreover, their corresponding examples are special cases of the following more general example.

\begin{example} \label{examplegeneral}
Let $(M^6,g)$ be a Ricci-flat Riemannian manifold and assume that there exist differential forms $\nu \in \Omega^1(M^6), \delta \in \Omega^2(M^6), \epsilon \in \Omega^3(M^6)$, which are closed and coclosed, satisfying $\|\nu\|_g^2=-1$, $\| \delta \|_g^2 =1$, and $\| \epsilon \|_g^2 =-1$, respectively.  Let $\widetilde{M}^{1,4}$ be a five-dimensional pp-wave with metric 
\[\tilde g= 2 \dd u \dd v - \sum_{i=1}^3 (\dd x^i)^2 + \frac{f_1(u)^2-f_2(u)^2+f_3(u)^2-f_4(u)^2}{6} \left(  \sum_{i=1}^3 (x^i)^2\right) \dd u^2,\] 
and set $\hat \alpha(u) = f_1(u) \vol_\rho, \hat \beta(u)=f_2(u)\dd x^1 \wedge \dd x^2, \hat \gamma(u) = f_3(u) \dd x^1, \hat \varpi(u) = f_4(u)$. If 
\begin{align*}
\Flu &= \dd u \wedge (\hat \alpha(u)+\hat \beta(u) \wedge \nu+\hat\gamma(u) \wedge \delta+\hat \varpi(u) \epsilon)  \\
&=\dd u \wedge (f_1(u) \dd x^1 \wedge \dd x^2\wedge \dd x^3 +f_2(u)\dd x^1 \wedge \dd x^2 \wedge \nu+ f_3(u) \dd x^1  \wedge \delta+f_4(u) \epsilon)
\end{align*}
then   $(\X^{1,10} = \widetilde{M}^{1,4} \times M^6, h=\tilde g + g, \Flu)$ is an eleven-dimensional bosonic supergravity background. Notice that when all but one of the functions $f_i$ vanish, then the example reduces to ones previously considered. 
We can specify the example even further by considering $M^6=\mathbb R^3 \times \Sigma$ with metric $g=-(\dd y^1)^2 - (\dd y^2)^2-(\dd y^3)^2-\nu$, where $\nu$ is a Ricci-flat positive definite metric on a three-dimensional manifold $\Sigma$, and $\nu= \dd y^3$, $\delta = \dd y^2 \wedge \dd y^3$, $\epsilon = \dd y^1 \wedge \dd y^2 \wedge \dd y^3$. 
\end{example}

%
%
%

\subsection{Results concerning the flux form $\Flu= \varphi \tilde{\alpha}+\tilde{\beta}\wedge\nu$} \label{sectalphabeta}
Now, set $\tilde \alpha=\dd u \wedge f(u,x) \vol_\rho$ and $\tilde \beta= \dd u \wedge \omega(u)$, where $f \in C^{\infty}(\mathbb R\times N^3)$ is a function,  $\omega(u) \in \Omega^2(N^3)$ is a 2-form depending smoothly on the parameter $u$, and $f, \omega(u),\nu,\varphi$ are nonzero.  Then we get the following statement regarding bosonic supergravity backgrounds with flux form $\Flu= \varphi \tilde{\alpha}+\tilde{\beta}\wedge\nu$.

\begin{prop}\label{gen2}
Let $(M^6,g)$ be a Riemannian Ricci-flat manifold, $\varphi$ a function on $M^6$ satisfying $\star_6 \dd \star_6 \dd \varphi = \kappa \lambda \varphi$ and set $\nu=\frac{1}{\kappa} \dd \varphi$, for a nonzero constant $\kappa$. Let also $(\widetilde{M}^{1,4}=\mathbb{R}\times N^3\times \mathbb{R},\tilde{g}=2\dd v\dd u+\rho+H\dd u^2)$ be a Walker manifold with $\rho$ Ricci-flat and $\partial_v(H)=0$.  Set as above $\tilde{\alpha}=\dd u \wedge f \vol_\rho$ and $\tilde \beta=\dd u \wedge \omega(u)$ with $\omega(u)=-\frac{1}{\lambda} \star_\rho \dd_N f$, for a nonzero constant $\lambda$, and assume that $f$ satisfies $\star_\rho \dd_N \star_\rho \dd_N f=-\kappa \lambda f$. Then $(\X^{1,10}=\widetilde{M}^{1,4}\times M^6,h=\tilde{g}+g,\Flu=\varphi \tilde{\alpha}+\tilde{\beta}\wedge\nu)$ is a bosonic supergravity background if and only if 
$$
\Delta_\rho H=\|\omega(u)\|_{\rho}^2 \| \nu \|_g^2 - f^2 \varphi^2 .
$$
\end{prop}
\begin{proof}
The proof is mainly based on Theorem \ref{alpha-beta-nu}. The Maxwell equation and  the closedness condition reduce to 
\begin{gather*} \dd_N \star_\rho \omega(u) =0,\;\;  \dd_N\omega(u)= \kappa f \vol_\rho, \;\;   \dd_N  f  = \lambda \star_\rho \omega(u), \;\;  \dd\nu=0, \;\;  \dd\varphi = \kappa \nu, \;\;  \dd \star_6 \nu=\lambda \star_6 \varphi.
\end{gather*} 
These equations are satisfied due to the definitions of $\nu$ and $\omega(u)$ and the two differential equations
\[\star_6 \dd \star_6 \dd \varphi = \kappa \lambda \varphi\,, \qquad \star_\rho \dd_N \star_\rho \dd_N f=-\kappa \lambda f \] constraining them. Next, the supergravity Einstein equation consists of the following system of equations:
\[
\left\{\begin{split}
\Ric_{h}(\tilde{X},\tilde{Y})&=-\frac{1}{2}\langle \tilde X\lrcorner \tilde{\alpha},\tilde Y\lrcorner \tilde{\alpha} \rangle_{\tilde{g}} \varphi^2-\frac{1}{2}\langle \tilde{X}\lrcorner \tilde{\beta},\tilde{Y}\lrcorner \tilde{\beta}\rangle_{\tilde{g}} \| \nu \|^2_{ g}\,, \quad \forall \ \tilde{X}, \tilde{Y} \in \Gamma(T\widetilde{M})\,,\\
0 &= \langle \tilde{\beta}, \tilde{X}\lrcorner \tilde{\alpha} \rangle_{\tilde{g}}\,, \quad \forall \ \tilde{X} \in \Gamma(T\widetilde{M})\,.
\end{split}\right.
\]
By   Lemma \ref{du2vanish} and by assuming that  $(\widetilde{M}^{1,4}, \tilde g)$ is of the form \eqref{gWalker} with $\rho$   Ricci-flat and $\partial_v(H)=0$, we see that the only nonzero component of the Ricci tensor has the form 
\[\Ric_h(\partial_u,\partial_u) = -\frac{1}{2} \left(f^2 \varphi^2 \| \vol_\rho\|^2_{\tilde g} +\| \omega(u)\|^2_{\tilde g} \| \nu\|_g^2 \right)= \frac{1}{2} \left(f^2 \varphi^2 -\| \omega(u)\|^2_{\tilde g} \| \nu\|_g^2 \right)\,,
\] 
while  for $\tilde \alpha$ and $\tilde \beta$ of the chosen form the equation $ \langle \tilde{\beta}, \tilde{X}\lrcorner \tilde{\alpha} \rangle_{\tilde{g}}=0$ holds automatically. 
 But then, by  \eqref{RicciWalker},  it turns out that the above relation is equivalent to the condition $\Delta_\rho H=\|\omega(u)\|_{\tilde g}^2 \| \nu \|_g^2 - f^2 \varphi^2$, which proves our claim.
\end{proof}
\begin{remark}
By assumption the function $f^2$ is nonzero, and $\|\omega(u)\|_{\tilde g}^2 =\|\omega(u)\|_{\rho}^2>0$, since $\rho$ is positive definite on 2-forms.  By Lemma \ref{sumofproducts} we see that either $\varphi$ and $\|\nu\|_g^2$ are both constant, or $f^2=C \|\omega(u)\|_{\tilde g}^2 $ for some nonzero constant $C$, which can only be positive. In the latter case we obtain 
\[
\|\nu\|_g^2 = C \varphi^2+D\,,
\]
for some constant $D$. In this equation the left-hand-side is negative. For this reason we  look for an example involving trigonometric functions. 
\end{remark}

\begin{example}\label{examplealphabeta}
Let $(M^6=\Ss^1\times \mathbb{R}^5, g)$ be a flat Riemannian manifold with coordinates $y^1,\dots,y^6$ and metric $g=-(\dd y^1)^2 - \sum_{i=2}^{6} (\dd y^i)^2$. Set 
\[
\varphi=\sin(y^1)\,,\quad \nu = \frac{1}{\kappa} \dd \varphi=\frac{1}{\kappa} \cos(y^1) \dd y^1\,.
\]
 Consider also a pp-wave $(\widetilde{M}^{1,4},\tilde g)$ with metric  $ \tilde g=2 \dd u \dd v -\sum_{i=1}^3 (\dd x^i)^2 + H(u,x) \dd u^2$. 
  Set  $f=\exp(x^1)$  and $\omega=\kappa \star_\rho \dd f= -\kappa \exp(x^1) \dd x^2 \wedge \dd x^3$. In terms of the notation in Proposition \ref{gen2}, we have  $\lambda=-1/\kappa$. 
Now, the supergravity Einstein condition gives
\begin{align*}
\Delta_\rho H &= \|\omega\|_{\tilde g}^2 \| \nu \|_g^2 - f^2 \varphi^2=  -\exp(2x^1)\,.
\end{align*}
This equation is satisfied if, for example, $H=\frac{1}{4} \exp(2x^1)$. Thus  $(\X^{1,10}=\widetilde{M}^{1,4} \times M^6, h=g+\tilde g, \Flu=\varphi \tilde \alpha + \tilde \beta \wedge \nu)$ is a bosonic supergravity background with these choices of $\varphi,\nu, f,\omega,H$, and then   the flux form $\Flu$  reads as
  $
 \Flu= \exp(x^1) \dd u \wedge \dd x^2 \wedge \dd x^3  \wedge \left(\sin(y^1)  \dd x^1 + \cos(y^1)\dd y^1\right).
  $
\end{example}

Let us now consider the case where $\dd\varphi=0$  (and $\kappa = 0$). By absorbing the constant into $\tilde \alpha$, we can without loss of generality assume that $\varphi= 1$.  As long as $\lambda \neq 0$, the 2-form $\omega(u)$ is determined by $f$ via the relation $\omega(u)=-\frac{1}{\lambda} \star_\rho \dd_N f$. This implies at once $\dd_N \star_\rho \omega(u)=0$. The equation $\dd_N \omega(u) =0$ is then equivalent to 
$\dd_N \star_\rho \dd_N f= 0$, or $\star_\rho \dd_N \star_\rho \dd_N f=0$. With this simplification we obtain the following statement. 

\begin{prop}\label{gen3}
Let $(M^6,g)$ be a Riemannian Ricci-flat manifold endowed with a   closed $1$-form $\nu\in\Omega^{1}(M)$ 
satisfying $\star_6 \dd \star_6 \nu=\lambda$, for some constant $\lambda \neq 0$. Let also $(\widetilde{M}^{1,4}=\mathbb{R}\times N^3\times \mathbb{R},\tilde{g}=2\dd v\dd u+\rho+H\dd u^2)$ be a Walker manifold with $\rho$ Ricci-flat and $\partial_v(H)=0$, and assume that   $f$ is a smooth function on $\mathbb R \times N^3$, such that 
\[
\star_\rho \dd_N \star_\rho \dd_N f=0\,.
\]
Set $\tilde{\alpha}=\dd u \wedge f \vol_\rho,\tilde \beta=\dd u \wedge \omega(u)$, with $\omega(u)=-\frac{1}{\lambda} \star_\rho \dd_N f$. Then $(\X^{1,10}=\widetilde{M}^{1,4}\times M^6,h=\tilde{g}+g,\Flu=\tilde{\alpha}+\tilde{\beta}\wedge\nu)$ is a bosonic supergravity background if and only if 
$$
\Delta_\rho H=\|\omega(u)\|_{\rho}^2 \| \nu \|_g^2 - f^2 .
$$
Consequently, $\|\nu\|_{ g}^2$ must be constant for such a bosonic supergravity background.
\end{prop}

Recall that by  Proposition \ref{compact}, the six-dimensional Riemannian manifold $M^6$ appearing in Proposition \ref{gen3}  must be  non-closed. Notice that examples \ref{examplegeneral} and \ref{examplealphabeta} can easily be modified to make $M^6$ a compact manifold, for example a flat six-dimensional torus. This is {\it not} the case for the following example.

\begin{example} \label{example3}
Consider the Riemannian manifold $M^6=(-L,L) \times \mathbb R^5$ with metric $g=-(\dd y^1)^2 - \sum_{i=2}^6 (\dd y^i)^2$. We set $\lambda=1$ and $\nu=-y^1 \dd y^1+ \sqrt{L^2-(y^1)^2} \dd y^2$, so that $\| \nu \|_g^2=-L^2$. Let $(\widetilde{M}^{1,4},\tilde g)$ be a pp-wave, that is  $\tilde g=2 \dd u \dd v -\sum_{i=1}^3 (\dd x^i)^2 + H(u,x) \dd u^2$.
If we set $f=x^1$ and $\omega=-\star_\rho \dd_N f=\dd x^2 \wedge \dd x^3$, then the Einstein equation is given by
\[\Delta_\rho H=\| \omega\|_\rho^2 \| \nu \|_g^2-f = -L^2-(x^1)^2.\] 
This equation is satisfied when, for example, $H=\frac{1}{12} (x^1)^4+\frac{L^2}{2} (x^1)^2$.  Thus $(\X^{1,10}=\widetilde{M}^{1,4} \times M^6,h=\tilde g+g, \Flu= \tilde \alpha + \tilde \beta \wedge \nu)$ is a bosonic supergravity background with these  choices of $\nu,f,\omega,H$. In this case, the flux form is given by $\Flu = \dd u \wedge \dd x^2 \wedge \dd x^3 \wedge (x^1 \dd x^1 - y^1 \dd y^1+\sqrt{L^2-(y^1)^2} \dd y^2).$
\end{example}

\section{Bosonic backgrounds involving K\"ahler manifolds and non-null flux forms} \label{other}
In this section we treat flux forms of type $\Flu=\tilde \gamma \wedge \delta$ and $\Flu=\theta$ for the case when $(M^6,g)$ is a K\"ahler manifold with K\"ahler form $\omega$. We will assume that the part of $\Flu$ taking values in $TM^6$ is given by $\omega$. More precisely, we consider the cases for which $\delta=\omega$ and $\theta=c \star_6 \omega$ for some nonzero constant $c \in \mathbb R$.

\subsection{Results concerning  $\Flu=\tilde{\gamma}\wedge\delta$}
Inspired by Corollary \ref{constantlength}, we consider bosonic supergravity backgrounds with $\Flu=\tilde{\gamma} \wedge \delta$ where $\| \tilde \gamma \|_{\tilde g}^2$ is not  (necessarily) constant, and hence not null. The following theorem follows directly from the bosonic supergravity equations (Proposition \ref{casesMC} and Proposition \ref{spck}). 

\begin{theorem}
Consider the Lorentzian manifold $\X^{1,10}=\wi{M}^{1,4} \times M^6$ with metric $h = \tilde g + g$ and a 4-form $\Flu=\tilde{\gamma} \wedge \delta$. Then $(\X^{1,10},h,\Flu)$ is a bosonic supergravity background if and only if $\delta$ is closed and coclosed, $\tilde \gamma$ is closed and satisfies 
\[ \dd \star_5 \tilde \gamma \wedge \star_6 \delta = \frac{\tilde \gamma \wedge \tilde \gamma \wedge \delta \wedge \delta}{2}\,,
\] 
and the following equations hold: 
\begin{equation}\label{Ric6}
\begin{split}
\Ric_{h}(X,Y)&=\left(\frac{\|\delta\|^2_g}{6}g(X,Y)-\frac{1}{2}\langle X\lrcorner \delta, Y\lrcorner \delta\rangle_g\right) \| \tilde{\gamma} \|^2_{ \tilde{g}}\,,\\
\Ric_{h}(\tilde{X},\tilde{Y})&=\left(\frac{\|\tilde{\gamma}\|^2_{\tilde{g}}}{6}\tilde{g}(\tilde{X},\tilde{Y})-\frac{1}{2}\langle \tilde{X}\lrcorner \tilde{\gamma},\tilde{Y}\lrcorner \tilde{\gamma}\rangle_{\tilde{g}}\right) \| \delta \|^2_{ g}\,.
\end{split}
\end{equation}
\end{theorem}

Now, Corollary \ref{constantlength} tells us that if $\|\tilde \gamma \|_{\tilde g}^2$ is not constant, then $M^6$ is a Ricci-flat almost Hermitian manifold. If we let $M^6$ be a K\"ahler manifold with K\"ahler form $\delta$, we get the following statement. 

\begin{prop} \label{prop61}
Let $(M^6,g,\delta)$ be a K\"ahler manifold and  let  $(\widetilde{M}^{1,4}, \tilde g)$ be a Lorentzian manifold. Assume also that $\tilde \gamma$ is a closed  $2$-form on $\widetilde{M}^{1,4}$. Then $(\X^{1,10}=\widetilde{M}^{1,4} \times M^6,h=g+\tilde g, \Flu=\tilde{\gamma} \wedge \delta)$ is a bosonic supergravity background if and only if $(M^6,g)$ is Ricci-flat, 
\[\dd \star_5 \tilde \gamma=\tilde \gamma \wedge \tilde \gamma\] 
and 
\begin{equation} \label{Ric7}
\Ric_{h}(\tilde{X},\tilde{Y})=\left(\frac{\|\tilde{\gamma}\|^2_{\tilde{g}}}{2}\tilde{g}(\tilde{X},\tilde{Y})-\frac{3}{2}\langle \tilde{X}\lrcorner \tilde{\gamma},\tilde{Y}\lrcorner \tilde{\gamma}\rangle_{\tilde{g}}\right) .
\end{equation}
\end{prop}
\begin{proof}
Let us first mention that  the K\"ahler form $\delta$ is closed and coclosed, and  that $\| \delta \|_g^2=3$. The equation  $\dd \star_5 \tilde \gamma \wedge \star_6 \delta = \frac{\tilde \gamma \wedge \tilde \gamma \wedge \delta \wedge \delta}{2}$  is in this case equivalent to $\dd \star_5 \tilde \gamma=\tilde \gamma \wedge \tilde \gamma$  because of the identity 
\begin{equation}
\label{eqn: kahler identity}
\star_6 \delta=\frac{1}{2} \delta \wedge \delta
\end{equation}
 on the K\"ahler form. The first of the equations \eqref{Ric6} is satisfied since $M^6$ is K\"ahler while the second reduces to \eqref{Ric7}. This completes the proof.
\end{proof}
This construction illustrates the well-known fact that Ricci-flat K\"ahler manifolds and, more specifically, Calabi-Yau manifolds play an important role as components of bosonic supergravity backgrounds.

\subsection{Results concerning the flux form $\Flu=\theta$ and K\"ahler-Einstein metrics}
By Corollary \ref{constantlength}, we know that $(\X^{1,10}=\widetilde{M}^{1,4} \times M^6, h=\tilde g + g, \Flu= \theta)$ is a solution of the supergravity Einstein equations, only if $\|\theta\|_{ g}^2$ is constant and $(\widetilde{M}^{1,4}, \tilde g)$ is an Einstein manifold with positive Einstein constant $\|\theta\|_{ g}^2/6$ (in our ``mostly minus'' convention).   By also taking into account Propositions \ref{casesMC} and \ref{spck}, we obtain the following theorem.
\begin{theorem} 
Consider the Lorentzian manifold $\X^{1,10}=\widetilde{M}^{1,4} \times M^6$ with metric $h=\tilde{g}+g$ and a 4-form $\Flu=\theta \in \Omega^4(M^6)$. Then  $(\X^{1,10},h,\Flu)$ is a bosonic supergravity background if and only if $\theta$ is closed and coclosed, $\| \theta \|_g^2$ is constant,  $(\widetilde{M}^{1,4},\tilde g)$ is an Einstein manifold with Einstein constant $\frac{1}{6}\|\theta\|_{g}^2$ and 
\begin{equation}
\label{eqn: theta Eins}
\Ric_{h}(X,Y)=\frac{\|\theta\|^2_g}{6}g(X,Y)-\frac{1}{2}\langle X\lrcorner \theta,Y\lrcorner \theta \rangle_g, \qquad \forall \ X,Y \in \Gamma(TM^6).
\end{equation}
\end{theorem}
As a consequence, we obtain the following proposition. 
\begin{prop}
\label{prop:thetakahler}
Consider the Lorentzian manifold $\X^{1,10}=\widetilde{M}^{1,4} \times M^6$, where $(\widetilde{M}^{1,4},\tilde g)$ is a Lorentzian  manifold and $(M^6,g,\omega)$ is a K\"ahler manifold. Then the triple $(\X^{1,10}, h=\tilde{g}+g, \Flu=\theta=c\star_6 \omega)$  is a bosonic supergravity background if and only if $(\widetilde{M}^{1,4},\tilde g)$  and $(M^6,g)$ are Einstein with   Einstein constants $\frac{1}{2}c^2$ and $-\frac{1}{2}c^2$, respectively.
\end{prop}
\begin{proof}
Let $M^6$ be a K\"ahler--Einstein manifold with complex structure $J$ and K\"ahler 2-form $\omega$, defined by $\omega_{ik}J^k_j=g_{ij}$. Obviously,  the 4-form  $\theta=c\star_6\omega$  is coclosed. By the identity \eqref{eqn: kahler identity},
we learn that the flux form is also closed. Hence the Maxwell equation and closedness condition are satisfied.
Recall now that $(M^6, g, \omega)$  admits an adapted orthonormal frame 
\[
e_1\,,\quad e_2:=Je_1\,,\quad e_3\,,\quad e_4:=Je_3\,,\quad e_5\,,\quad e_6:=Je_5\,,
\]
such that each $\omega_{ij}=0$ if $(i,j)\notin \{(1,2),(2,1),(3,4),(4,3),(5,6),(6,5)\}$. Then, by \eqref{eqn: kahler identity}, we see that 
\[
\theta_{ijkl}=\omega_{ij}\omega_{kl}-\omega_{ik}\omega_{jl}-\omega_{il}\omega_{kj}\,.
\]
This relation can be used to  compute $\frac{1}{3!}\theta_{iabc}\theta_{j}^{\ abc}$, where with $g$ we may raise or lower the indices.
In particular, having in mind the relation $\omega_{ik}J^k_j=g_{ij}$, it is straightforward  to verify that
\[
\langle\theta_i,\theta_j\rangle_h=\frac{1}{3!}\theta_{iabc}\theta_{j}^{\ abc}=2g_{ij}=\frac{2}{3}\|\theta\|^2_g g_{ij}\,.
\]
Note that $\star_6$ is an isometry and we have $c^2\|\omega\|^2_g=\|\theta\|^2_g$. Hence \eqref{eqn: theta Eins} together with the second equation of \eqref{Rictheta} reduces to the following system of equations:
\begin{equation}
\begin{cases}
\Ric_{h}(X,Y)&=-\frac{c^2}{2}g(X,Y)\,,\\
\Ric_{h}(\tilde{X},\tilde{Y})&=\frac{c^2}{2}\tilde{g}(\tilde{X},\tilde{Y})\,.
\end{cases} \label{EinCond}
\end{equation}
 This completes our proof.
\end{proof}
By Corollary \ref{useforflags} the bosonic background $(\X^{1, 10}=\wi{M}^{1,4} \times M^6, h=\tilde g+g, \Flu=\theta=c\star_{6}\omega)$ discussed above must have constant positive scalar curvature given by $\mathsf{Scal}_{h}=\frac{1}{6}\|\theta\|^2_g=\frac{c^2}{2}$. Moreover, we see that  $(\wi{M}^{1,4},\tilde g)$ has positive scalar curvature while  $(M^6,g)$ has negative scalar curvature.

\begin{corol}\label{mainKE}
Let $(\wi{M}^{1,4},\tilde g)$ be a Lorentzian Einstein manifold with Einstein constant $\frac{1}{2} c^2$ and let $(M^6,g,\omega)$ be an Einstein K\"ahler manifold with with Einstein constant $-\frac{1}{2}c^2$. Then the triple \[(\wi{M}^{1,4} \times M^6, \tilde g + g, \Flu= c \star_6 \omega)\] is a bosonic supergravity background. 
\end{corol}

\begin{remark}
Several bosonic backgrounds of the type appearing in Corollary \ref{mainKE} are already known. In \cite{Pope89} it was shown that $(\mathsf{AdS})_{5}\times M^6$,
where $(M^6,g,\omega)$ is  a K\"ahler manifold, defines a bosonic background with flux form $\Flu=\frac{1}{2} c \omega \wedge \omega$ (which is proportional to $\star_6 \omega$), if and only if $(M^6,g)$ is Einstein with negative scalar curvature (in ``mostly minus'' setting). In particular, the equations (4) and (5) in \cite{Pope89} are the same as \eqref{EinCond}, up to a scalar and signature convention. Thus Corollary \ref{mainKE} can be viewed as a  slight generalization of the backgrounds found in \cite{Pope89}, in the sense that $\wi{M}^{1,4}$ is here allowed to be any Lorentzian Einstein manifold with positive Einstein constant (in the ``mostly minus'' setup), not only $(\mathsf{AdS})_{5}$.
\end{remark}
In order to get other bosonic supergravity backgrounds from Corollary \ref{mainKE}, we need Lorentzian Einstein manifolds $(\wi{M}^{1,4}, \tilde g)$ with positive scalar curvature that are different from $(\mathsf{AdS})_{5}$. The class of Lorentzian Einstein-Sasakian manifolds provides us with many candidates. 
 
Lorentzian Einstein-Sasakian structures were  studied in \cite{Baum2, Bohle, Baum3}   under the geometric perspective of Killing and twistor spinors on Lorentzian manifolds. Such  manifolds admit a cone characterization and, in the simply connected case, a spin structure (\cite[Lem.~12]{Bohle}), as in the Riemannian case.  Moreover, there is  an   analogue of  the well-known construction of Riemannian Einstein-Sasakian manifolds (see for example \cite{Baum}), which  provides   Lorentzian Einstein-Sasakian manifolds in terms of
circle bundles over K\"ahler-Einstein  manifolds of negative scalar curvature  (\cite[Lem.~14]{Bohle}). 
 
 Lorentzian Einstein-Sasakian manifolds with positive  (in the ``mostly minus'' convention) Einstein constant also  occur  within the framework of {\it $\eta$-Einstein Sasakian geometry}, and we refer to \cite{Boyer} for many details and notions that we omit. 
 A particularly important result, from our perspective,  is that every {\it negative Sasakian manifold $M^{2n+1}$} admits   a Lorentzian Einstein-Sasakian structure with positive Einstein constant $2n$, in the ``mostly minus'' convention (\cite[Cor.~24]{Boyer}). In particular, the 5-sphere $\Ss^5$ admits infinitely many different Lorentzian Einstein-Sasakian structures, with  Einstein constant $4$ (all of them are inhomogeneous).  Each of these can be used as the five-dimensional Lorentzian manifold $\widetilde{M}^{1, 4}$ in Corollary \ref{mainKE}, providing us with infinitely many  decomposable backgrounds (carrying the same flux form). 
 
\begin{example} \label{5sphere}
Let $(\wi{M}^{1,4},\tilde g)$ be the $5$-sphere $\Ss^5$ with an Einstein metric coming from any of the infinitely many Lorentzian Einstein-Sasakian structures mentioned above with Einstein constant $4$, and let $(M^6,g,\omega)$ be any Einstein K\"ahler manifold with Einstein constant $-4$. Then the triple 
\[ (\wi{M}^{1,4} \times M^6, \tilde g + g , \Flu= 2 \sqrt{2} \star_6 \omega)\]
is a bosonic supergravity background. 
\end{example}
 
 The connected sum $\sharp k (\Ss^2\times\Ss^3)$ also admits Lorentzian Einstein-Sasakian metrics for any integer $k\geq 1$ (see \cite{Gomez, Boyer}), giving another class of Lorentzian Einstein-Sasakian manifolds that can potentially be used as ingredients in bosonic supergravity backgrounds. Example \ref{5sphere} highlights the appearance of Lorentzian Einstein-Sasakian geometries in supergravity, and more  applications of such structures in eleven-dimensional supergravity are described in \cite{FS}.  For further details on  the applications of five-dimensional Lorentzian manifolds in certain supergravity theories, the reader may consult the recent work \cite{BF21} and the references therein.

\begin{remark}
Some of the  bosonic backgrounds of the type appearing in Corollary \ref{mainKE} are symmetric.  For example, if $\wi{M}^{1,4}=(\mathsf{AdS})_{5}$ and $M^6$ is one of the symmetric spaces $\mathbb{CP}^3=\SU(4)/ \mathsf{U}(3)$ or $\mathsf{Gr}_{+}(2,5)$, endowed with their respective  homogeneous K\"ahler-Einstein metrics, then we obtain decomposable  symmetric  backgrounds (see also \cite[Sect. 4.4]{Fig2}).   To construct decomposable, homogeneous but non-symmetric, $(5, 6)$-supergravity backgrounds, we may  use the full flag manifold $\bb{F}=\SU(3)/T_{\mathsf{max}}$, see for example \cite{Sak1} for the corresponding  K\"ahler-Einstein metrics. Also,  in \cite{Oog} the reader can find    families of non-supersymmetric bosonic backgrounds with non-relativistic symmetry based on products involving $(\mathsf{AdS})_{5}$.
Finally, it is worth mentioning that \eqref{eqn: kahler identity} also holds for a strictly nearly K\"ahler manifold,    and $\|\theta\|_g^2=c^2 \|\omega\|_g^2$ is a constant as well, see  \cite[Cor.~2.7]{MNS05}. However, for a strictly nearly K\"ahler manifold, the K\"ahler form is not closed. Thus,   for this case the 4-form $\Flu$ chosen above is not coclosed, and can not serve as a flux form.  
\end{remark}

\section{Conclusion}
The aim of this paper was to find new bosonic supergravity backgrounds by searching for decomposable $(5,6)$-solutions to the bosonic supergravity equations (1.1). Following the ideas proposed in \cite{CG}, where they were applied to find decomposable $(6,5)$-solutions, we analyzed the bosonic supergravity equations for a variety of different types of flux 4-form $\Flu$. We singled out some special cases that we analyzed more carefully in order to get concrete new supergravity backgrounds.

For several different types of null 4-forms we found supergravity backgrounds whose 5-dimensional Lorentzian part was a special type of Ricci-isotropic Walker manifold. These results, which can be found in Section 5, are very close in nature to those of \cite{CG}, and one may conjecture that similar results and examples can be found for decomposable $(m,11-m)$-solutions with other choices of $m$. A couple of places where we differed from \cite{CG} were Corollary 3.5 and Proposition 5.11 (and the corresponding Example 5.12) which contain insights that were not used in \cite{CG}. These solutions may support supersymmetries, but we leave this investigation for a later paper.

In Section 6 we investigated two particular types of flux forms that were not null. Here we relied on some known results that are very specific to the particular pair of dimensions, $(5,6)$. For example, by assuming that the six-dimensional Riemannian component is a K\"ahler manifold, and by choosing a 4-form $\Flu$ that depended explicitly on the K\"ahler form, we showed that any Einstein K\"ahler manifold with Einstein constant $-c^2/2$  (in the ``mostly minus'' convention) can be paired with any Lorentzian Einstein manifold with Einstein constant $c^2/2$ to give an eleven-dimensional bosonic supergravity background. One background of this type that was already known is $(\mathsf{AdS})_5 \times M^6$ which was treated in \cite{Pope89} (the symmetric ones among these are also listed in Section 4.4 of \cite{Fig2}), and we describe infinitely many more: From \cite{Boyer} we know that the 5-sphere admits infinitely many different Lorentzian Einstein-Sasakian structures with Einstein constant 4, and each of them give rise to a bosonic supergravity background when paired with an Einstein K\"ahler manifold with Einstein constant $-4$.

\medskip
\noindent {\bf Acknowledgements:}  
H.C. thanks NSFC for partial support under grants No. 11521101 and No. 12071489.   I.C. and E.S. acknowledge full support by Czech Science Foundation  via the project GA\v{C}R No.~19-14466Y.  It is also our pleasure to  thank A. Galaev (UHK) for helpful discussions.


\medskip
\noindent {\bf Authors declarations:}  
The authors declare that they have no conflicts  of interest. 

\medskip
\noindent {\bf Author contributions:}  
Hanci Chi: Writing - original draft (equal); writing - review and editing (equal). Ioannis Chrysikos:  Writing - original draft (equal); writing - review and editing (equal). Eivind Schneider:  Writing - original draft (equal); writing - review and editing (equal).

\medskip
\noindent {\bf Data availability:}  
Data sharing is not applicable to this article as no new data were created or analyzed in this study.

\bibliographystyle{unsrt}

\end{document}